\documentclass[reqno]{amsproc}
\usepackage{amsmath,amssymb}
\usepackage{thmtools}
\usepackage{mathtools}
\usepackage[numbers]{natbib}

\usepackage{hyperref}
\usepackage{cleveref}

\usepackage[dvipsnames]{xcolor}

\hypersetup{breaklinks=true,
    linkcolor=Bittersweet,
    urlcolor=blue,
    citecolor=blue,
    colorlinks=true,
pdfborder={0 0 0}}

\numberwithin{equation}{section}

\makeatletter
\renewcommand{\bibsection}{\@bibtitlestyle}
\makeatother

\def\R{\mathbb{R}}
\def\E{\mathbb{E}}
\def\tE{\tilde{\mathbb{E}}}
\def\eps{\varepsilon}
\DeclareMathOperator{\tr}{tr}
\DeclareMathOperator{\vect}{vec}

\DeclareMathOperator{\rank}{rank}
\DeclareMathOperator{\Rem}{Rem}

\DeclareMathOperator{\loo}{LOO}
\DeclareMathOperator{\alo}{ALO}
\DeclareMathOperator{\diag}{diag}
\DeclareMathOperator*{\argmin}{arg\,min}

\DeclareMathOperator{\Err}{Err}

\declaretheorem[name=Theorem,numberwithin=section]{theorem}
\declaretheorem[name=Proposition,sibling=theorem]{proposition}
\declaretheorem[name=Lemma,sibling=theorem]{lemma}
\declaretheorem[name=Corollary,sibling=theorem]{corollary}

\declaretheorem[name=Remark,style=remark,numberwithin=section]{remark}
\declaretheorem[name=Example,style=remark,numberwithin=section]{example}

\crefname{assumption}{assumption}{assumptions}

\begin{document}
\title[Simultaneous analysis of ALO-CV and mean-field inference I]{Simultaneous analysis of approximate leave-one-out cross-validation
and mean-field inference I}
\author{Pierre C. Bellec $^\dagger$}
\thanks{$^\dagger$:
Rutgers University.
Email:
\href{mailto:pierre.bellec@rutgers.edu}{pierre.bellec@rutgers.edu}}

\begin{abstract}
Approximate Leave-One-Out Cross-Validation (ALO-CV) is a method that has
been proposed to estimate the generalization error of a regularized
estimator in the high-dimensional regime where dimension and sample size
are of the same order, the so-called ``proportional regime''.  A new
analysis is developed to derive the consistency of ALO-CV for
non-differentiable regularizers under Gaussian covariates and
strong convexity. Using disintegration over affine fibers and
weighted Gaussian integration by parts, the
difference between the ALO-CV weights and their counterparts in mean-field
inference is shown to be small.  Combined with upper bounds between the
mean-field inference estimate and the leave-one-out quantity, this
provides a proof that ALO-CV approximates the leave-one-out quantity up to negligible error terms.  Linear models with square loss, robust
linear regression and single-index models are explicitly treated.
\end{abstract}
\maketitle

\tableofcontents

\section{Introduction to ALO-CV 
and mean-field inference}
Consider iid observations $(x_i,y_i)_{i=1,...,n}$
with $x_i\in \R^p$ and $y_i$ valued in some fixed set $\mathcal Y$.
For a test function of interest $g:\R\times \mathcal Y \to \R$,
we are interested in the generalization error with respect to
$g$, namely
$$
\Err_g(b)
= \E\Bigl[
g(b^Tx_{new}, ~ y_{new})
\Bigr],
$$
for any candidate vector $b\in\R^p$,
where $(x_{new},y_{new})$ is an independent copy of $(x_i,y_i)$.
The goal of the present paper is to analyze simultaneously
two statistical methods to estimate the generalization error
of an estimate
\begin{equation}
    \label{hat_b}
\hat b = \argmin_{b\in \R^p}
\sum_{i=1}^n L_{y_i}(x_i^Tb)
+ R(b)
\end{equation}
for convex loss functions $L_{y_i}: \R \to \R$
and a convex penalty $R:\R^p\to \R$.
In this case, the generalization error can be written as
the conditional expectation
\begin{equation}
    \label{err_b_hat}
\Err_g(\hat b) 
= \E\Bigl[
g(\hat b^Tx_{new}, ~ y_{new})
\mid
(x_i,y_i)_{i=1,...,n}
\Bigr].
\end{equation}

Above and throughout the paper, $n$ is the sample size
and $p$ the dimension. The ratio $p/n$ is kept of constant
order as $n,p\to+\infty$, with
\begin{equation}
    \label{proportional_regime}
\delta \le \frac p n  \le \delta^{-1}
\end{equation}
for some constant $\delta\in (0,1)$.
This is commonly referred to as the proportional regime,
or proportional asymptotics,
which has been the focus of much recent research in high-dimensional
statistics and optimization. We refer
to \cite{montanari2018mean} for an introduction to the topic.
Let us survey two methods that have been proposed
to estimate
\eqref{err_b_hat} in the regime \eqref{proportional_regime}.

\subsection{Approximate leave-one-out cross-validation}

A recent series of works 
\cite{rad2018scalable,
xu2019consistent,
wang2018approximate,
rad2020error,
auddy2024approximate,
zou2024theoretical}
have proposed and analyzed
the Approximate Leave-One-Out Cross-Validation (ALO-CV)
method, which we now review. The starting point is the
leave-one-out estimate
\begin{equation}
b^i = \argmin_{b\in \R^p}
\sum_{l\in [n] \setminus \{i\}}
L_{y_l} (x_l^T b) + R(b),
\label{min_bi}
\end{equation}
where the $i$-th observation is left out compared to \eqref{hat_b}.
The leave-one-out estimate of \eqref{err_b_hat} is then given by
\begin{equation}
    \label{loo}
    \loo
    \coloneq
\frac 1 n \sum_{i=1}^n g(x_i^T b^i, y_i).
\end{equation}
Error bounds of the estimation error of \eqref{err_b_hat}
by the leave-one-out estimate \eqref{loo}
and consistency are obtained
in \cite{xu2019consistent,zou2024theoretical}.
However, the leave-one-out estimate \eqref{loo}
requires one to solve the minimization problem $n$ times.
ALO-CV attempts to approximate \eqref{loo}
while solving the minimization problem only once
on the full dataset.
To do so, the ALO-CV estimate \cite{rad2018scalable} is constructed
by performing one Newton step, starting from $\hat b$,
to minimize \eqref{min_bi}. The resulting approximation of $b^i$
is
\begin{equation}
    \label{newton}
\tilde b_{\textsc{newton}}^i
= \hat b + \Bigl(\sum_{k\in[n]\setminus\{i\}}x_k L_{y_k}''(x_k^T\hat b) x_k^T + \nabla^2 R(\hat b)\Bigr)^{-1} x_i L_{y_i}'(x_i^T\hat b)
\end{equation}
which is well-defined as long as $L_{y_k}$ is twice differentiable
with second derivative $L_{y_k}''$
and the Hessian $\nabla^2 R$ exists and is positive-definite
at $\hat b$.
Replacing $b^i$ by $\tilde b_{\textsc{newton}}^i$ in the leave-one-out estimate
\eqref{loo}, and using the Sherman-Morrison-Woodbury matrix identity
to obtain
$$
x_i^T
\Bigl(\sum_{k\in[n]\setminus\{i\}}x_k L_{y_k}''(x_k^T\hat b) x_k^T + \nabla^2 R(\hat b)\Bigr)^{-1} x_i
=
\frac{x_i^T \hat A x_i}{1 - x_i^T \hat A x_i L_{y_i}''(x_i^T\hat b)}
$$
where $\hat A = (\sum_{i=1}^n x_i L_{y_i}''(x_i^T\hat b) x_i^T + \nabla^2 R(\hat b))^{-1}$. This leads to the ALO-CV estimate of \eqref{err_b_hat} given by
\begin{equation}
    \label{alo}
    \alo
    \coloneq
\frac1n
\sum_{i=1}^n g\Bigl(x_i^T \hat b + L_{y_i}'(x_i^T\hat b)
\frac{x_i^T \hat A x_i}{1 - x_i^T \hat A x_i L_{y_i}''(x_i^T\hat b)}
,~ y_i\Bigr).
\end{equation}
Each step above is detailed in \cite{rad2018scalable,auddy2024approximate}.
This construction has been extended to some regularizers that are not
differentiable such as $R(b) = \lambda \|b\|_1$
\cite[\S2.2]{wang2018approximate,wang2018approximate_fast_tuning,auddy2024approximate,rad2018scalable}.

ALO-CV is appealing because its construction is simply based on
the leave-one-out estimate and the above Newton step,
with no assumption on the underlying data-generating process.
It is computationally efficient compared to the leave-one-out estimate
\eqref{loo}. It does not suffer from the sample size bias
seen when using 5- or 10-fold cross-validation to approximate
\eqref{err_b_hat} \cite[Figure 1]{xu2019consistent,rad2018scalable}:
Estimating \eqref{err_b_hat} by 5-fold cross-validation
produces approximations of the generalization error of an estimator trained
on $4/5$ of the data, which is not the same as the generalization
error of \eqref{hat_b} trained on the full dataset.

Proving the success of ALO-CV for estimating
\eqref{loo} or \eqref{err_b_hat}, for instance consistency,
was initially obtained 
for Gaussian designs under differentiability and smoothness conditions
on the loss functions $L_{y_i}$ and the penalty $R$
\cite{rad2018scalable,xu2019consistent}.
The differentiability and smoothness conditions were recently
relaxed to allow for the elastic-net penalty
under stability conditions on the supports $S^i = \{j\in[p]: b^i_j \ne 0\}$
of \eqref{min_bi}
\cite[Theorem 3.2]{auddy2024approximate},
and such stability condition is shown to hold for the square loss
$L_{y_i}(t)=(y_i-t)^2/2$ and the elastic-net penalty
under isotropic Gaussian $x_i$
\cite[\S4]{auddy2024approximate}.
The work \cite{auddy2024approximate} leverages these support stability
results and smoothing arguments to prove consistency of ALO-CV.

This paper presents a completely different argument, without
relying on smoothing arguments or support stability conditions.
The main contribution of the paper is a simple argument to prove
the consistency of ALO-CV for non-smooth regularizers under Gaussian
design, by directly relating the ALO-CV estimate \eqref{alo}
to results obtained in mean-field asymptotics that we now describe.

\subsection{Mean-field asymptotics}
The last decade has seen the development of mean-field asymptotics,
which aims to study the asymptotic behavior of high-dimensional
estimators such as \eqref{hat_b}
in the proportional regime \eqref{proportional_regime},
see for instance
\cite{donoho2009message,bayati2012lasso,el_karoui2013robust,stojnic2013framework,donoho2016high,el_karoui2018impact,thrampoulidis2018precise,miolane2018distribution,celentano2020lasso,loureiro2021capturing}
and \cite{montanari2018mean} for an introduction to the topic.
The picture that emerges is that the behavior of the empirical
distribution of $\hat b$ or of $(x_i^T\hat b)_{i\in[n]}$ can 
be precisely described by a few deterministic scalars
that are the solution of a system of nonlinear equations.
The precise definition of these deterministic quantities is not
relevant for the purpose of the present paper and the aforementioned
papers provide numerous examples.
For the purpose of the present paper, let us only describe 
one convergence result that is typical of mean-field asymptotics:
If $x_i\sim N(0, I_p)$,
if the test function $g$ satisfies some regularity conditions,
and if $y_i\mid x_i$ follows a single-index model in the sense
that $y_i= F(x_i^Tw, E_i)$ for some deterministic $F(\cdot,\cdot)$ and
$w\in\R^p$ with $\|w\|_2=1$ as well as some
external randomness $E_i$ independent of $x_i$,
then
\begin{equation}
    \label{mean_field}
\frac 1 n
\sum_{i=1}^n g\Bigl(
x_i^T\hat b + \gamma_* L_{y_i}'(x_i^T\hat b),
y_i\Bigr)
\to^P
\E\Bigl[
g(a_*U + \sigma_*G,y)
\Bigr]
\end{equation}
where $(\gamma_*,a_*,\sigma_*)$ are deterministic scalars found
by solving the aforementioned system of equations,
and in the right-hand side $(U,y)$ is equal in distribution
to $(x_i^Tw,y_i)$ and $G\sim N(0,1)$ is independent of $(U,y)$.
See for instance \cite[Theorem 2]{loureiro2021capturing}.
Such approximation could be the starting point in deriving
estimates of the generalization error \eqref{err_b_hat},
although in practice $(\gamma_*,a_*,\sigma_*)$ are unknown:
they need to be estimated from the data, or at the very least
approximations such as \eqref{mean_field} need to hold
with $\gamma_*$ replaced by an observable quantity $\hat \gamma$.
Estimating $(\gamma_*,a_*,\sigma_*)$ with observable quantities
has been the focus of \cite{bellec2022observable} by the author,
where observable quantities that estimate $(\gamma_*,a_*,\sigma_*)$ are
developed and referred to as observable adjustment.
We describe in the next subsection
a version of \eqref{mean_field} where $\gamma_*$ in the left-hand side 
is replaced by an observable quantity $\hat \gamma$,
which is a key intermediate
ingredient for the analysis of ALO-CV of the present paper.

\subsection{Mean-field inference}

One family of results in mean-field asymptotics focuses on
inference regarding the components $\beta_j^*$ of a
ground-truth $\beta^*$ in regression models,
or of the components $w_j$ of a ground-truth $w$ in single-index models.
This is referred to as de-biasing, where the initial
estimate $\hat b_j$ provided by the $j$-th component of \eqref{hat_b}
requires a bias correction after which the de-biased estimate
is approximately normal and centered at $\beta_j^*$ in regression
\cite{bellec_zhang2019debiasing_adjust,bellec_zhang2019second_poincare,celentano2020lasso},
or at a shrunk version of $w_j$ in single index models
\cite[Section 4.1]{bellec2022observable}.

Another family of results useful for inference lets us
construct an estimate of the generalization error \eqref{err_b_hat}.
The principle here is similar to de-biasing:
for a given coordinate $i\in[n]$,
start with the initial value $x_i^T\hat b$ and find
a de-biasing correction that makes 
$x_i^T\hat b + \textsf{correction}$ approximately normal,
see for instance \cite[Theorem 4.3]{bellec2022observable}.
Let us state formally one version of such results that will
be useful for the present paper.

We will use the notation $X\in\R^{n\times p}$
for the matrix with rows $x_1,...,x_n$ and
\begin{equation}
    \label{D}
D=\diag(L_{y_i}''(x_i^T\hat b)_{i\in[n]})\in\R^{n\times n}.
\end{equation}
When $L_y'$ is Lipschitz but not differentiable everywhere,
$L_y''(t)$ denotes any measurable generalized derivative of $L_y'$ at $t$;
under convexity of $L_y$ and the 1-Lipschitz property of $L_y'$,
this choice belongs to $[0,1]$.
Furthermore, $e_i\in\R^n$ is the $i$-th canonical basis
vector in $\R^n$, and $e_j\in\R^p$ the $j$-th canonical
basis vector in $\R^p$.
We say that the regularizer is $(n\mu,\Sigma)$-strongly convex
if
$b\mapsto R(b) - \mu n \|\Sigma^{1/2}b\|_2^2/2$ is convex, or equivalently
\begin{equation}
\inf_{u\in\partial R(b)}
\inf_{u'\in\partial R(b')}
(b-b')^T (u-u')
\ge n\mu \|\Sigma^{1/2}(b-b')\|_2^2,
\label{mu-n-strongly-convex}
\end{equation}

\begin{proposition}
    \label{prop:mean_field}
    Consider iid $(x_i,y_i)_{i\in[n]}$ with
    $x_i\sim N(0, \Sigma)$, and $w\in\R^p$ with $\E[(w^Tx_i)^2]\in\{0, 1\}$
    such that
    $(I_p - \Sigma ww^T)x_i$ is independent of $(x_i^Tw, y_i)$.
    Assume that for all values of $y\in\mathcal Y$,
    the loss $L_y(\cdot)$ is convex, differentiable and 1-Lipschitz,
    with derivative $L_y'(\cdot)$ being 1-Lipschitz.
    Assume that the regularizer $R$ is $(n\mu,\Sigma)$-strongly convex in the sense of \eqref{mu-n-strongly-convex}
    and that $R$ is minimized at 0.
    Then for fixed values of $(y_i)_{i\in[n]}$ the mapping
    $$
    (x_{ij})_{i\in[n],j\in[p]} \mapsto
    \hat b
    $$
    is almost everywhere differentiable
    with $\frac{\partial}{\partial x_{ij}}\hat b
    =
    \hat A(-e_j L_{y_i}'(x_i^T\hat b) - X^TDe_i \hat b_j)$
    for some matrix $\hat A\in\R^{p\times p}$
    with 
    \begin{equation}
        \label{bound-A}
        \|\Sigma^{1/2}\hat A\Sigma^{1/2}\|_{op} \le (n\mu)^{-1}.
    \end{equation}
    Furthermore, with the leave-one-out estimate
    $b^i$ in \eqref{min_bi},
    \begin{equation}
    \sum_{i=1}^n
    \E\Bigl[
    \Bigl(
    x_i^T\hat b + \tr[\Sigma\hat A] L_{y_i}'(x_i^T\hat b)
    -
    x_i^Tb^i
    \Bigr)^2
    \Bigr]
    \le
    C(\mu,\delta)
    \label{desired}
    \end{equation}
    for some constant depending on $\mu,\delta$ only
    provided that $p/n$ satisfies \eqref{proportional_regime}.
    If additionally $g$ satisfies
    $\sup_{y\in\mathcal Y}|g(x,y) - g(x',y)|\le |x-x'|(1+|x|+|x'|)$
    then
    \begin{equation}
        \label{bound_mean_field_LOO}
    \E
    \Big|
    \loo
    -
    \frac1n
    \sum_{i=1}^n
    g\Bigl(
    x_i^T\hat b + \tr[\Sigma\hat A] L_{y_i}'(x_i^T\hat b),
    ~
    y_i
    \Bigr)
    \Big|
    \le \frac{C(\mu,\delta)}{\sqrt n}.
    \end{equation}
\end{proposition}

Since $x_i\sim N(0, I_p)$ is independent of $b^i$,
\eqref{desired} shows that $x_i^T b^i$ plus the additive correction
$\tr[\Sigma\hat A] L_{y_i}'(x_i^T\hat b)$ is approximately normal;
informally
$$
x_i^T\hat b  + \tr[\Sigma \hat A] L_{y_i}'(x_i^T\hat b)
\approx N(0,1) \|b^i\|_2 
$$
and $\|b^i\|_2$ in the right-hand side can be replaced by 
$\|\hat b\|_2$ up to a small error term by \eqref{loo_bound_hat_b_bi}
below.
The statement of \Cref{prop:mean_field} and its proof
are variants of Theorem 4.3 in \cite{bellec2022observable}
with a minor variation on how the normal random variable
is constructed. Here the standard normal random variable
$x_i^Tb^i/\|b^i\|_2$ is obtained by leaving one observation out,
in \cite[Theorem 4.3]{bellec2022observable} it is obtained
by replacing the $i$-th row of $X\in\R^{n\times p}$ by an independent copy.
Since \Cref{prop:mean_field} does not appear in the literature in the form stated,
let us provide a short proof.

\begin{proof}[Proof of \Cref{prop:mean_field}]
    Let us prove the case $\Sigma=I_p$.
    The case $\Sigma\ne I_p$ can be deduced from the case $\Sigma=I_p$
    by a change of variable as explained in
    \cite[Appendix B]{bellec2021derivatives}
    or
    \cite[Appendix C.2]{bellec2022observable}.

    The derivative formula, the existence of the matrix $\hat A$
    and the upper bound on $\|\hat A\|_{op}$ are proved
    in \cite[Proposition 3.1]{bellec2022observable}
    (or in \cite{bellec2021derivatives} in linear regression).
    Let $P=I_p - ww^T$.
    Applying the consequence
    \cite[(2.6)]{bellec_zhang2018second_stein}
    of the Second Order Stein formula
    to the Gaussian vector
    $Px_i$ conditionally on $(x_i^Tw,y_i)$ gives
    $$
    \E\Bigl[
        B_i^2
    \Bigr]
    \le
    \E\Bigl[
    \|P(\hat b - b^i)\|_2^2
    +
    \sum_{j=1}^p
    \|
    \frac{\partial \hat b}{\partial x_{ij}}
    \|_2^2
    \Bigr]
    \text{ where }
    B_i = 
        x_i^TP\hat b - \sum_{j=1}^p e_j^T P \frac{\partial\hat b}{\partial x_{ij}}
    -
    x_i^T P b^i.
    $$
    Using the derivative formula to bound the second term,
    $$
    \sum_{i=1}^n
    \sum_{j=1}^p
    \|
    \frac{\partial \hat b}{\partial x_{ij}}
    \|_2^2
    \le 2 \|\hat A\|_F^2 \sum_{i=1}^n L_{y_i}'(x_i^T\hat b)^2
    + 2 \|\hat A X^TD\|_F^2 \|\hat b\|_2^2.
    $$
    For the first term $\|P(\hat b - b^i)\|^2$, we use \cite[Lemma 3 item 4]{auddy2024approximate}\cite[Appendices 2-3]{el_karoui2018impact}:
    \begin{equation}
        \label{loo_bound_hat_b_bi}
        n\mu \|\hat b - b^i\|_2 \le \|x_i\|_2  |L_{y_i}'(x_i^T\hat b)|.
    \end{equation}
    A short proof is as follows: write the KKT conditions
    $\sum_{l=1}^n x_l L_{y_l}'(x_l^T\hat b)\in - \partial R(\hat b)$
    for $\hat b$,
    the KKT conditions
    $\sum_{l\ne i} x_l L_{y_l}'(x_l^Tb^i)\in - \partial R(b^i)$
    for $b^i$,
    and use the strong convexity property
    $(\hat b - b^i)^T(\partial R(\hat b) - \partial R(b^i)) \ge n\mu \|\hat b - b^i\|_2^2$:
    \begin{align*}
    n\mu \|\hat b - b^i\|_2^2 
    & \le 
    (\hat b - b^i)^T(\partial R(\hat b) - \partial R(b^i))
  \\&=
      (\hat b - b^i)^T\Bigl(-x_i L_{y_i}'(x_i^T\hat b)\Bigr)
      -\sum_{l\ne i}
      \Bigl(x_l^T \hat b -x_l^T b^i\Bigr)\Bigl(
      L_{y_l}'(x_l^T\hat b)
      -
      L_{y_l}'(x_l^Tb^i)
      \Bigr).
    \end{align*}
    See \cite[Section 4.2]{bellec2020out_of_sample} where similar arguments
    are used to establish smoothness properties of $\hat b$.
    Since the rightmost sum is non-negative by convexity of $L_{y_l}$,
    \eqref{loo_bound_hat_b_bi} holds by the Cauchy-Schwarz
    inequality.
    Following the idea in \cite[(C.3)]{bellec2020out_of_sample},
    writing the strong convexity property of $R$ at
    $\hat b$ and $0$ for the element of $\partial R(\hat b)$
    given by the KKT conditions further gives
    $
    n\mu \|\hat b - 0\|_2^2
    \le
    (\hat b - 0)^T(\partial R(\hat b) - \partial R(0))
    =
    - \hat b^T \sum_{i=1}^n x_i L_{y_i}'(x_i^T \hat b)
    $
    so that for instance
    \begin{align}
        \label{bound_b_hat}
    &n\mu \|\Sigma^{1/2}\hat b\|_2 \le \|X\Sigma^{-1/2}\|_{op} \sqrt n \max_{l\in[n]} |L_{y_l}'(x_l^T\hat b)|,
    \\\text{by the same argument }
    &n\mu \|\Sigma^{1/2}b^i\|_2 \le \|X\Sigma^{-1/2}\|_{op} \sqrt n \max_{l\in[n]\setminus\{i\}} |L_{y_l}'(x_l^Tb^i)|.
        \label{bound_b_i}
    \end{align}
    Let $C_i = 
    x_i^T\hat b + \tr[\hat A] L_{y_i}'(x_i^T\hat b)
    -
    x_i^Tb^i$ as in the left-hand side of the desired result
    \eqref{desired}.
    Using the derivative formula, with $U_i = x_i^Tw$, we find
    \begin{align*}
        C_i-B_i
        &=
        U_i w^T(\hat b - b^i)
        +\tr[ww^T\hat A] L_{y_i}'(x_i^T\hat b)
        - \hat b^T P \hat A x_i L_{y_i}''(x_i^T\hat b).
    \end{align*}
    By the triangle inequality
    $\E[\sum_{i=1}^n(C_i+D_i)^2]^{1/2}
    \le
    \E[\sum_{i=1}^nC_i^2]^{1/2}
    +
    \E[\sum_{i=1}^n D_i^2]^{1/2}$
    several times, we find
    that $\E[\sum_{i=1}^n C_i^2]^{1/2}$ is bounded from above by
    $$
    \E\Bigl[\sum_{i=1}^n B_i^2\Bigr]^{\frac12}
    +
    \E\Bigl[\sum_{i=1}^n (U_iw^T(\hat b - b^i))^2\Bigr]^{\frac12}
    +
    \E\Bigl[\|\hat A\|_{op}^2\sum_{i=1}^n L_{y_i}'(x_i^T\hat b)^2\Bigr]^{\frac12}
    +
    \E\Bigl[
    \|\hat b^TP\hat AX^TD\|_F^2
    \Bigr]^{\frac12}.
    $$
    The claim \eqref{desired} is proved by
    combining these upper bounds with
    \eqref{bound-A}, \eqref{loo_bound_hat_b_bi}, \eqref{bound_b_hat},
    inequalities
    $|L_{y_i}'(x_i^T\hat b)| \le 1$
    and $|L_{y_i}''(x_i^T\hat b)| \le 1$ (so that $\|D\|_{op}\le 1$)
    by assumption on $L_{y_i}$, the fact that
    $\E[\|X\|_{op}^2]\le(\sqrt n + \sqrt p)^2+1$
    thanks to
    \cite[Theorem II.13]{DavidsonS01},
    and the inequality
    $\E[\sum_{i=1}^n U_i^2 \|x_i\|_2^2] \le n(p-1) + 3n$
    since $x_i\sim N(0, I_p)$ and $U_i=x_i^T w$.
    For \eqref{bound_mean_field_LOO}, by assumption on the test
    function $g$
    the left-hand side is bounded from above by
    $\E \frac{1}{n}\sum_{i=1}^n|C_i| (1+|x_i^T b^i| + |C_i + x_i^Tb^i|)$
    so that the Cauchy-Schwarz inequality combined with
    \eqref{desired} and \eqref{bound_b_i} gives \eqref{bound_mean_field_LOO}.
\end{proof}

\subsection{Examples of matrices $\hat A$}
For concreteness, let us provide the expression of $\hat A$
in a few examples of interest.
Assuming that $L_{y_i}(\cdot)$ is twice-differentiable,
and let $D\in\R^{n\times n}$ be the diagonal matrix \eqref{D}
with $D_{ii}=L_{y_i}''(x_i^T\hat b)$.
\begin{example}
    \label{example:twice}
    Assume $R$ is twice differentiable
    with Hessian $\nabla^2 R(b)$
    having smaller eigenvalue bounded from below by $\nu n$ for some $\nu\ge 0$,
    for every $b\in\R^p$. Then,
    $\hat A  = (X^TD X + \nabla^2 R(\hat b))^{-1}.$
\end{example}
\begin{example}
    \label{example:enet}
    If $R(b) = \lambda \|b\|_1 + n \nu\|b\|_2^2/2$
    is the elastic-net for $\lambda, n\nu \ge 0$,
    $$
    \hat A_{\hat S, \hat S}  = (X_{\hat S}^TD X_{\hat S} + n\nu I_{\hat S,\hat S})^{-1},
    \qquad
    \hat A_{j,k}=0 \text{ if } j\notin \hat S \text{ or } k\notin \hat S
    $$
    almost surely 
    where $\hat S = \{j\in[p]: \hat b_j\ne 0\}$ is the support of $\hat b$.
    Above $X_{\hat S}$ is the submatrix of $X$ with columns
    indexed in $\hat S$ and $I_{\hat S,\hat S}$ is the identity matrix
    of size $|\hat S|$.
\end{example}
\begin{example}[group-lasso]
    \label{example:group}
    If $\lambda_1,...,\lambda_k>0$, $\nu\ge 0$,
    a partition $(G_1,...,G_K)$ of $[p]$ is given
    and $R(b) = \sum_{k=1}^K \lambda_k \|b_{G_k}\|_2 + n\nu\|b\|_2^2/2$ 
    then with
    $\hat T = \{k\in [K] : \|\hat b_{G_k}\|_2 >0 \}$
    and $\hat S = \cup_{k\in \hat T} G_k$, almost surely
    $$
    \hat A_{\hat S, \hat S}  = 
    \Bigl(X_{\hat S}^TD X_{\hat S} +
        n\nu I_{\hat S,\hat S}
        +
        \sum_{k\in \hat T}
        \frac{\lambda_k}{\|\hat b_{G_k}\|_2}
        \Bigl(I_{G_k,G_k} - \frac{\hat b_{G_k}\hat b_{G_k}^T}{\|\hat b_{G_k}\|_2^2}\Bigr)
    \Bigr)^{-1}
    $$
    and $\hat A_{j,k}=0$ if $j\notin \hat S$ or $k\notin \hat S$.
\end{example}

These formulae are obtained by differentiating the KKT conditions.
For instance, in the case of twice-differentiable $R$,
differentiating $\sum_{l=1}^n x_l L_{y_l}'(x_l^T\hat b)  + \nabla R(\hat b)= 0$
with respect to $x_{ij}$ gives
$e_j L_{y_i}'(x_i^T\hat b) + (\nabla^2R(\hat b) + X^TDX) \frac{\partial}{\partial x_{ij}}\hat b + X^TD e_i \hat b_j = 0$.
The derivative formula in \Cref{prop:mean_field} thus holds
with $\hat A$ given by \Cref{example:twice}.
More care is needed in the case of \Cref{example:enet,example:group}
as the penalty is non-differentiable. In such non-differentiable cases,
the above formulae are obtained by first showing that
almost surely the active set $\hat S$ is locally constant;
once this is established, the KKT conditions can be differentiated
as in the twice-differentiable case.
We refer to \cite[Section 4]{bellec_zhang2019second_poincare}
for such arguments applied to the lasso and group-lasso.
In \Cref{prop:mean_field} the matrix $\hat A$ is shown to exist,
however no general closed-form expression for $\hat A$ is available
for general regularizer $R$ at this point. The above cases are the most
common though, and the formulae for $\hat A$ in these cases are 
explicit and given above.

For ALO-CV, the construction with the Newton step
\eqref{newton} leading to \eqref{alo} is only valid
when the Hessian $\nabla^2 R$ exists and is positive-definite.
In this case matrix $\hat A$ coincides with the expression
in \Cref{example:twice}. For non-differentiable penalty functions
$R$ such as the L1 or elastic-net penalty, the ALO-CV construction
can be extended by approximation
\cite[Section 2.2]{rad2018scalable}
leading to a well-defined matrix $\hat A$ in \eqref{alo}.
In all cases where both
\begin{itemize}
    \item a closed-form formula for the matrix $\hat A$
        in the expression $\frac{\partial}{\partial x_{ij}}\hat b$
        in \Cref{prop:mean_field} is available, and
    \item the ALO-CV construction \eqref{alo} has been extended
        for some matrix $\hat A$,
\end{itemize}
the expression of $\hat A$ in the derivative formula of \Cref{prop:mean_field}
and the expression of $\hat A$ in the ALO-CV construction \eqref{alo}
coincide.

In existing analysis of ALO-CV for non-differentiable regularizer
\cite{auddy2024approximate} where the penalty is the elastic-net,
the major difficulty
comes from controlling the difference between
the support $\hat S$ of $\hat b$ defined in \Cref{example:enet}
and the support $S^i$ of the leave-one-out vector $b^i$ in \eqref{min_bi}.
\citet{auddy2024approximate} manage to achieve such control
between the supports $S^i$ and $\hat S$
for the square loss, elastic-net penalty and isotropic Gaussian designs,
and as a consequence obtains the consistency of ALO-CV in this case.
This paper develops an alternative argument to derive the consistency
of ALO-CV for non-differentiable regularizer,
by relating ALO-CV directly to mean-field approximations
such as \eqref{bound_mean_field_LOO},
without trying to control the difference between 
the supports 
$\hat S$ and $S^i$ or other discrepancies between $\hat b$ and $b^i$.

\subsection{Reconciling ALO-CV and mean-field inference}
An observation that motivated the present study is that
ALO-CV provides at first glance a different picture than the usual
mean-field asymptotics results of the proportional regime.
ALO-CV has been shown \cite{rad2018scalable,xu2019consistent,auddy2024approximate}
to consistently estimate the leave-one-out estimate \eqref{loo}, i.e.,
$$
\alo
=
\frac1n
\sum_{i=1}^n g\Bigl(x_i^T \hat b + L_{y_i}'(x_i^T\hat b)W_i ,~ y_i\Bigr)
\approx
\loo
\text{ where }
W_i \coloneq \frac{x_i^T \hat A x_i}{1 - x_i^T \hat A x_i L_{y_i}''(x_i^T\hat b)}
$$
under suitable assumptions.
On the other hand, \Cref{prop:mean_field} shows that
\begin{equation}
\frac1n
\sum_{i=1}^n g\Bigl(x_i^T \hat b + L_{y_i}'(x_i^T\hat b)\tr[\Sigma \hat A] ,~ y_i\Bigr)
\approx
\loo.
\label{mean_field_approx_LOO_informal}
\end{equation}
Typical proportional asymptotics results such as \eqref{mean_field}
(e.g., Theorems 1 and 2 in \cite{loureiro2021capturing}) also suggest
that the weight in front of $L_{y_i}'(x_i^T\hat b)$ should be
independent of $i$ to obtain the mean-field limit
of the generalization error \cite[Theorem 1, eq. (2.5)]{loureiro2021capturing}.

Mean-field inference thus points to the weight $\tr[\Sigma \hat A]$
in \eqref{mean_field_approx_LOO_informal} while ALO-CV
points to the weights $W_i$ above.
This raises the question of whether the weights $(W_i)_{i\in[n]}$
concentrate around $\tr[\Sigma \hat A]$ or the deterministic value
$\gamma_*$ in \eqref{mean_field} (it is shown
that $\tr[\Sigma \hat A] - \gamma_*\to^P 0$ in 
\cite[Lemma 13]{koriyama2024precise} in regression for separable loss and
penalty, so that $\tr[\Sigma \hat A]$ and $\gamma_*$ can be used interchangeably).

Relating ALO-CV to estimates of the form \eqref{mean_field_approx_LOO_informal}
(with a single weight independent of $i$)
has appeared before in \cite{xu2019consistent}.
In \cite[equation (17)]{xu2019consistent}, the ALO-CV weights $W_i$ are shown,
under suitable
conditions that include twice-differentiability
of both the loss and penalty,
to concentrate uniformly in $i\in[n]$ around
a data-driven scalar (independent of $i$)
found by solving a nonlinear equation
\cite[(8)-(9)]{xu2019consistent}.
However, the argument in \cite{xu2019consistent} relies on
twice-differentiability of the loss and penalty and requires
that the second derivatives are themselves $\alpha$-Holder continuous
for some $\alpha\in(0,1]$.

When
comparing the weight $\tr[\Sigma \hat A]$ in \eqref{mean_field_approx_LOO_informal}
to the weights $W_i$ in ALO-CV,
a resembling known approximation that comes to mind is
\begin{equation}
    \label{approx_trAtrV_df}
\tr[\Sigma \hat A] \approx \frac{\tr[X\hat AX^TD]}{\tr[D - DX\hat AX^TD]}
= \frac{
    \sum_{i=1}^n x_i^T\hat A x_i L_{y_i}''(x_i^T\hat b)
}{
    \sum_{i=1}^n
 (1-
    x_i^T\hat A x_i L_{y_i}''(x_i^T\hat b)
    )
L_{y_i}''(x_i^T\hat b)
},
\end{equation}
see \cite[Section 5]{bellec2021derivatives} in linear regression
and \cite{bellec2022observable} in single-index models.
Without the two sums in the numerator and the denominator,
the rightmost fraction would be the same as the $W_i$ weights
from ALO-CV.
The approximation \eqref{approx_trAtrV_df} in \cite{bellec2021derivatives,bellec2022observable} is proved for nondifferentiable regularizers,
so \eqref{approx_trAtrV_df} provides a little hope 
towards reconciling ALO-CV with \eqref{mean_field_approx_LOO_informal}
under no smoothness assumption on the regularizer $R(\cdot)$ in \eqref{hat_b}.
However, the proofs of \eqref{approx_trAtrV_df} in \cite{bellec2022observable,bellec2021derivatives} do leverage the two summations in the numerator and denominator
of \eqref{approx_trAtrV_df}.
The main technical contribution of the present paper is a novel argument,
different from the leave-one-out bounds in \cite{xu2019consistent,auddy2024approximate} or the Stein formulae argument to prove \eqref{approx_trAtrV_df}
in \cite{bellec2021derivatives,bellec2022observable},
to obtain $W_i\approx \tr[\Sigma \hat A]$ for non-differentiable regularizers.

\subsection{Notation}
The norms $\|\cdot\|_2, \|\cdot\|_1$ are the Euclidean and $\ell_1$ norms
of vectors. If none is specified, $\|\cdot\|$ denotes the Euclidean norm.
For matrices, $\|\cdot\|_{op}$ is the operator norm (largest singular value)
and $\|\cdot\|_F$ is the Frobenius norm.
We use the set notation $[n]=\{1,...,n\}$ and similarly for $[p]$.

\section{Square loss and rotational invariance}

Let us first consider the square loss $L_{y_i}(t) = (y_i-t)^2/2$
in \eqref{hat_b}, so that the minimization problem becomes
\begin{equation}
    \label{hat_b_square_loss}
\hat b = \argmin_{b\in \R^p}
\|y - Xb\|_2^2 / 2 + R(b),
\end{equation}
and assume that $(X,y)$ has continuous distribution.
It is known that for any fixed $X$, the mapping
$y\mapsto X\hat b$ is 1-Lipschitz (see, e.g., \cite{bellec2016bounds})
so that its Jacobian
$$
H \coloneq \frac{\partial}{\partial y} X\hat b
\in \R^{n\times n}
$$
exists almost everywhere.
The matrix $H$ is sometimes referred to the ``hat'' matrix and
appears in the literature on Stein's unbiased risk estimate
\cite{stein1981estimation}.
In this case, the matrix $\hat A$ in \Cref{prop:mean_field}
(whose existence is granted by \cite[Theorem 1]{bellec2021derivatives}
for strongly-convex regularizer)
is related to $H$ by $H=X\hat A X^T$.

In this case, the ALO-CV weights are given by
$$
W_i = \frac{H_{ii}}{1-H_{ii}}.
$$
On the other hand,
\cite[Section 5]{bellec2021derivatives} proves that
under a $n\mu$-strong convexity assumption on $R$
and \eqref{proportional_regime},
\begin{equation}
    \label{result-trAtrV-df}
\E\Bigl[ \Big|\tr[\Sigma\hat A] - \frac{\tr[H]}{n-\tr[H]} \Big| \Bigr] \le 
\frac{C(\mu,\delta)}{\sqrt n}
\end{equation}
so that $\tr[\Sigma\hat A]$ and $\tr[H]/(n-\tr[H])$ can be used interchangeably
in results of the form \eqref{bound_mean_field_LOO}
or \eqref{mean_field_approx_LOO_informal}.
The following result shows that the approximation
$W_i\approx \tr[H]/(n-\tr[H])$ holds uniformly in $i\in[n]$.

\begin{proposition}
    \label{prop:square_loss}
    Consider the regularized least-squares estimate
    \eqref{hat_b_square_loss} where $R:\R^p\to\R$ is convex,
    and assume a linear model $y=X\beta^* + \eps$
    where $\eps\sim N(0,\sigma^2 I_n)$ is independent of $X$ and
    the rows $x_i$ of $X$ are iid $N(0,\Sigma)$.
    Then
    \begin{equation}
    \max_{i\in [n]}|H_{ii} - \tr[H]/n|
    \le
    c \sqrt{\log(n)/n}
    \label{conclusion_H_ii_tr-H_n}
    \end{equation}
    with probability at least $1-c/n$
    for some absolute constant $c>0$.
    Furthermore, if $R$ in \eqref{hat_b_square_loss}
    is $(\mu n,I_p)$-strongly convex then in the same event
    \begin{equation}
        \label{conclusion_W_i_squares_loss}
    \Big|
    W_i
    - \frac{\tr[H]}{n-\tr[H]}
    \Big|
    =
    \Big|
    \frac{H_{ii}}{1-H_{ii}}
    - \frac{\tr[H]}{n-\tr[H]}
    \Big|
    \le c\sqrt{\log(n)/n} \Bigl(1 + \|X\|_{op}^2/(n\mu)\Bigr)^2.
    \end{equation}
\end{proposition}
With \eqref{result-trAtrV-df} and \eqref{conclusion_W_i_squares_loss},
we see that
$$
\tr[\Sigma \hat A],
\qquad \frac{\tr[H]/n}{1-\tr[H]/n},
\qquad W_i=\frac{H_{ii}}{1-H_{ii}},
$$
can be used interchangeably in results of the form \eqref{mean_field_approx_LOO_informal} up to a small error term of order $\sqrt{\log(n)/n}$.
The proof given below relies mostly on the rotational invariance of the square loss
and the Gaussian distribution. Smoothness of the regularizer is irrelevant.
Inequality \eqref{conclusion_H_ii_tr-H_n} only requires
convexity of the regularizer,
while for \eqref{conclusion_W_i_squares_loss} we additionally
assume strong convexity as this is a simple condition that lets us
control $1/(1-\tr[H]/n)$. Controlling
$1/(1-\tr[H]/n)$ without strong convexity is possible, e.g.,
\cite{celentano2020lasso}, however this requires more involved arguments
and different assumptions such as small-enough sparsity of $\beta^*$
\cite{celentano2020lasso,bellec2020out_of_sample}.

Let us explain with an example how \eqref{conclusion_H_ii_tr-H_n}
can be used with existing results in mean-field asymptotics
to draw conclusions about ALO-CV.
By Theorems 7 and 8 in \cite{celentano2020lasso},
for any 1-Lipschitz function $\phi:\R^n\to\R$, 
if the sparsity of $\beta^*$ is small enough (in the sense of the
phase transition studied in \cite{celentano2020lasso}),
the Lasso $\hat b$ satisfies the convergence in probability
\begin{equation}
\phi\Bigl(\frac{y-  X \hat b}{(1-\tr[H]/n)\sqrt n}\Bigr)
-
\E\Bigl[\phi\Bigl(\frac{\tau_* Z}{\sqrt n}\Bigr)\Bigr]
\to^P 0
\label{phi}
\end{equation}
where $Z\sim N(0,I_n)$ and $\tau_*^2$ is such that
of $\sigma^2 + \|\Sigma^{1/2}(\hat b - \beta^*)\|_2^2 - \tau_*^2 \to^P 0$
in the linear model $y=X\beta^* + \eps$ where $X$ has iid $N(0,\Sigma)$
rows and $\eps\sim N(0,\sigma^2 I_n)$ is independent of $X$.
Consider a test function $g:\R^2\to\R$ 
in \eqref{err_b_hat} of the form $g(a, y) = \varphi(a-y)$,
so that the ALO-CV estimate \eqref{alo} is
$$
\alo
=
\frac1n \sum_{i=1}^n \varphi\Bigl(x_i^T\hat b-y_i + W_i (x_i^T\hat b-y_i)\Bigr)
=
\frac1n \sum_{i=1}^n \varphi\Bigl((1+W_i) (x_i^T\hat b-y_i)\Bigr)
$$
and $1+W_i = 1/(1-H_{ii})$.
Theorem 8 in \cite{celentano2020lasso} shows that
$\tr[H]/n=\|\hat b\|_0/n$ is bounded away from 1
in the sense that $1 - \tr[H]/n > c_0 >0$ with high-probability,
so that by \eqref{conclusion_H_ii_tr-H_n} the scalar
$H_{ii}$ is also bounded away from 1 uniformly in $i\in[n]$.
Since $\varphi$ is Lipschitz, \eqref{conclusion_H_ii_tr-H_n} gives
$$
\Big|
\alo
- 
\frac1n\sum_{i=1}^n
\varphi\Bigl(
\frac{x_i^T\hat b - y_i}{1-\tr[H]/n}
\Bigr)
\Big|
\lesssim \sqrt{\log(n)/n}
$$
with high-probability.
The function $\phi(u) = \frac1n\sum_{i=1}^n\varphi(-u_i\sqrt n)$
is 1-Lipschitz by the Cauchy-Schwarz inequality.
Applying \eqref{phi} with this function and combining with the preceding bound gives
$$
\Big|
\alo
-
\E[\varphi(\tau_* G)]
\Big|
\to^P 0
$$
where $G\sim N(0,1)$.
Since 
$\sigma^2 + \|\Sigma^{1/2}(\hat b - \beta^*)\|_2^2 - \tau_*^2 \to^P 0$
by \cite{celentano2020lasso},
the quantity
$\E[\varphi(\tau_* G)]$ is the same as \eqref{err_b_hat}
up to a negligible error term converging to 0 in probability.
This proves rigorously that if the sparsity of $\beta^*$ is small enough
in the sense of the phase transition in \cite{celentano2020lasso},
then ALO-CV consistently estimates the generalization error
\eqref{err_b_hat} for test functions of the form
$g(a,y)=\varphi(a-y)$ for 1-Lipschitz functions $\varphi$
(a natural choice being $\varphi(u) = |u|$).

\begin{proof}[Proof of \Cref{prop:square_loss}]
    Let us realize $(X,\eps)$
    as $(QG, Qz)$ where $Q\in\R^{n\times n}$ is random
    orthogonal matrix in $O(n)$ distributed according to the Haar measure,
    and $Q,G,z$ are mutually independent, with $G$ having iid $N(0,\Sigma)$ rows
    and $z\sim N(0,\sigma^2 I_n)$.
    Then $y=Q(G\beta^*+z)$, and
    \begin{equation}
    \hat b =
    \argmin_{b\in \R^p}
    \|y - Xb\|_2^2/2 + R(b)
    =
    \argmin_{b\in \R^p}
    \|z + G\beta^* - Gb\|_2^2/2 + R(b)
    \label{argue_rotational_inv}
    \end{equation}
    since orthogonality gives
    \[
    \|y-Xb\|_2 = \|Q(z+G\beta^*-Gb)\|_2 = \|z+G\beta^*-Gb\|_2
    \]
    for all $b\in\R^p$. The rotated objective depends only on $(G,z)$
    and is therefore independent of $Q$.
    Furthermore, with $\bar y = z + G\beta^* = Q^T y$
    and $\bar H = \frac{\partial }{\partial \bar y} G\hat b$
    we have by the chain rule
    \begin{align*}
    H 
    &= \tfrac{\partial }{\partial y} (X\hat b) &&\text{ by definition of $H$}
  \\&= Q \tfrac{\partial }{\partial y} (G\hat b) &&\text{ by linearity }
    \\&= Q \tfrac{\partial }{\partial \bar y} (G\hat b) Q^T 
      &&\text{ by the chain rule}
    \\&= Q \bar H Q^T.
    \end{align*}
    By construction, $\bar H$ is independent of $Q$ since this is the hat
    matrix in the linear model $\bar y = G\beta^* + z$
    of the estimator with design matrix $G$ and response $\bar y$
    (since both $G$ and $z$ are independent and $Q$, the estimator and
    its hat matrix from data $(G, \bar y)$ is independent of $Q$ a well).
    Furthermore, $\bar H$ is symmetric with eigenvalues in $[0,1]$
    by \cite[Remark 3.3]{bellec_zhang2018second_stein}.
    We can now write $H_{ii} = u_i^T \bar H u_i$ where $u_i = Q^T e_i$ and $e_i\in\R^n$ is the $i$-th canonical basis vector.
    Since $u_i = Q^Te_i$ for $e_i\in\R^n$ the $i$-th canonical basis vector
    is uniformly distributed on the sphere and independent of $\bar H$,
    by the Hanson-Wright inequality
    (see for instance
    \cite[Proposition 6.4 with $t=\log(n^2)$]{simchowitz2018tight})
    we find since $\bar H$ is symmetric with eigenvalues in $[0,1]$ that
    $$
    \forall i\in[n],
    \qquad
    \mathbb P[
    |H_{ii} - \tr[\bar H]/n| \le
    c \sqrt{\log(n)/n}
    ] \ge 1 - c /n^2
    $$
    for some absolute constant $c>0$.
    The union bound over $i\in[n]$ proves \eqref{conclusion_H_ii_tr-H_n}
    since $\tr\bar H = \tr H$.
    By the argument  \cite[between inequalities (D.7) and (D.8)]{bellec2021derivatives},
    if $R$ is $(n\mu, I_p)$-strongly convex we have
    $\|H\|_{op} \le 1/((n\mu)/\|X\|_{op}^2 + 1)$, so that
    using $H_{ii}\le \|H\|_{op}$ and $\tr[H]/n \le \|H\|_{op}$ we get
    $$
    \max\Bigl\{
    \frac{1}{1-\tr[H]/n}
    ,
    \frac{1}{1-H_{ii}}
    \Bigr\}
    \le \frac{(n\mu)/\|X\|_{op}^2 + 1}{(n\mu)/ \|X\|_{op}^2}
    =
    1 + \|X\|_{op}^2/(n\mu).
    $$
    Due to
    $|\frac{H_{ii}}{1-H_{ii}} - \frac{\tr H/n}{1-\tr[H]/n}|
    \le
    \frac{|H_{ii} - \tr[H]/n|}{1-H_{ii}}
    + \frac{\tr[H]/n |H_{ii}-\tr[H]/n|}{(1-H_{ii})(1-\tr[H]/n)}
    $
    this proves \eqref{conclusion_W_i_squares_loss}.
\end{proof}

\section{Beyond rotational invariance: probabilistic results}
\label{sec:proba}

Beyond the square loss, we cannot use rotational invariance
as in the previous section.
The probabilistic result that lets us control the error in the approximation
$W_i \approx \tr[\hat A \Sigma]$ is the following.
The object of interest is a matrix of the form
$$
A=(X^TDX + \Sigma^{1/2}H\Sigma^{1/2})^{-1}
$$
where $(D,H,\Sigma)$ are deterministic matrices,
and $X$ is given by
    \begin{equation}
        \label{X_is_QZP}
        X = QZ\Sigma^{1/2}P + Q\bar U + \bar V
    \end{equation}
    where $Z$ has iid $N(0,1)$ entries
    and $Q,P,\bar U,\bar V$ are deterministic matrices
    of appropriate dimensions.
    If $X$ has iid $N(0, \Sigma)$ rows, this is the affine Gaussian law that arises as the regular
    conditional law associated with the affine constraints
    \begin{equation}
    X(I_p-P)=\bar U, \qquad (I_n-Q)X=\bar V.
    \label{affine_constraints}
    \end{equation}
The matrix $H$ is a positive definite
matrix that will be related to the regularizer $R$,
for instance for twice-continuously differentiable regularizer,
$H$ below can be thought of as
$\Sigma^{-1/2}\nabla^2 R(\hat b)\Sigma^{-1/2}$, in which case the matrix $A$ above is the same as $\hat A$ in \Cref{example:twice}.
Of course, $\nabla^2 R(\hat b)$ in \Cref{example:twice} is random, and it is unclear at first glance that studying a deterministic matrix $H$ in the above expression for $A$ can be useful. We will explain in \Cref{sec:4,sec:glm} that
a disintegration argument allows us to work on events of the form
\eqref{affine_constraints}
where $\hat b$,
$D$ and the relevant regularizer quantities are deterministic.
Although the fibers are affine, conditioning on fitted quantities
need not give a Gaussian law: a coarea Jacobian can produce a density
tilt. The Gaussian result below and its weighted extension in
\Cref{cor:tilted-proba} account for this distinction. After disintegration,
these results for deterministic $H$ become useful to control
the error in the approximation $W_i \approx \tr[\hat A \Sigma]$.

\begin{theorem}
    \label{thm:proba}
    Let $\mu>0$ and $\delta$ as in \eqref{proportional_regime},
    and let $\Sigma$ be symmetric positive definite.
    Consider deterministic matrices $P\in\R^{p\times p}$
    and $Q\in\R^{n\times n}$
    such that $\Sigma^{1/2}P\Sigma^{-1/2}$
    and $Q$ are orthogonal projections.
    Let deterministic $\bar U\in\R^{n\times p}$ and
    $\bar V\in\R^{n\times p}$ satisfy the compatibility conditions
    \begin{equation}
        \label{compatibility}
        \bar U P=0,
        \qquad
        Q\bar V=0,
        \qquad
        (I_n-Q)\bar U = \bar V(I_p-P).
    \end{equation}
    Let $Z\in\R^{n\times p}$ have iid standard normal entries and
    let $X$ be defined by \eqref{X_is_QZP}.
    Let $H\in\R^{p\times p}$
    with $H - n\mu I_p$ symmetric positive semi-definite.
    Let $D\in\R^{n\times n}$ be diagonal with entries in $[0,1]$.
    Let $A=(X^TDX + \Sigma^{1/2}H\Sigma^{1/2})^{-1}$ and define
    \begin{equation}
    \Rem_i =
    x_i^TAx_i
    - \tr[A\Sigma](1-D_{ii} x_i^TAx_i).
    \label{Rem_i}
    \end{equation}
    Let $d_P=\dim\ker P$ and $d_Q = \dim\ker Q$.
    Then
    \begin{equation}
       \sum_{i=1}^n \E\Bigl[ \Rem_i^2 \Bigr]
       \le
       C(\mu, \delta) \Bigl(1 + d_P^2 + d_Q\Bigr)
       \E\Bigl[
       \Bigl( 1 + \frac{\|X\Sigma^{-1/2}\|_{op}}{\sqrt n}\Bigr)^4
       \Bigr].
       \label{conclusion33}
    \end{equation}

\end{theorem}
\begin{proof}
    \renewcommand{\tE}{\E}
    Assume without loss of generality that $\Sigma=I_p$,
    since the $\Sigma\ne I_p$ case can be deduced from
    the $\Sigma=I_p$ case by the change of variable
    $$(X,A,P,Q,H,\bar U,\bar V)\mapsto
    (X\Sigma^{-1/2}, \Sigma^{1/2}A\Sigma^{1/2},
    \Sigma^{1/2}P\Sigma^{-1/2}, Q,H,
    \bar U\Sigma^{-1/2},\bar V\Sigma^{-1/2}).$$
    In the case $\Sigma=I_p$, \eqref{X_is_QZP} becomes
    \begin{equation}
    X = QZP + Q\bar U + \bar V.
    \label{representation}
    \end{equation}
    The compatibility conditions \eqref{compatibility} imply
    $X(I_p-P)=\bar U$, $(I_n-Q)X=\bar V$ and $QXP=QZP$.
    Equivalently, with $M=Q\bar U + \bar V$ to denote the mean and
    $\vect(\cdot)$ denoting the vectorization operator that stacks the
    columns of a matrix into a vector, the distribution of $\vect(X)$ is
    characterized by
    \begin{equation}
        \text{Law}(\vect(X)) =
        N(\vect(M), P \otimes Q)
        \label{dist-z}
    \end{equation}
    where $\otimes$ is the Kronecker product.

The inequality given after (2.7) in \cite{bellec_zhang2018second_stein}
states that if $z\sim N(m, \Lambda)$ in $\R^q$ and $f:\R^q\to\R^q$
is locally Lipschitz then
$$
\E[((z-m)^Tf(z) - \tr[\Lambda \nabla f(z)])^2]
\le \E[ \|\Lambda^{1/2}f(z)\|_2^2 + \|\Lambda^{1/2}\nabla f(z)\Lambda^{1/2}\|_F^2].
$$
If $\Lambda$ is an orthogonal projection matrix,
applying this to $\tilde f(z) = \Lambda f(z)$ gives
\begin{equation}
    \label{SOS_tilde_f}
\E\Bigl[\Bigl((z-m)^T\tilde f(z) - \tr[\nabla \tilde f(z)]\Bigr)^2\Bigr]
\le \E\Bigl[ \|f(z)\|_2^2 +\sum_{u=1}^q \|\frac{\partial}{\partial z_u} f(z)\|_2^2\Bigr].
\end{equation}
thanks to $\|\Lambda\|_{op}\le 1$ in the right-hand side,
and because $\Lambda$ can be absorbed into $\tilde f$ in the second
term of the left-hand side.
The term $\tr[\nabla \tilde f(z)] = \sum_{u=1}^q \frac{\partial}{\partial z_u} \tilde f_u(z)$ becomes the divergence of $\tilde f$.

We will use the property $\vect(ABC)=(C^T\otimes A)\vect(B)$ of the vectorization operator for matrices $A,B,C$ of compatible dimensions.
Here, take $q=np$, $z =\vect(X)$
with distribution \eqref{dist-z}, and the function
$\tilde f(\vect(X)) = \vect(Q e_ie_i^TXA P) = (P\otimes Q) \vect(e_ie_i^TXAP)$.
This gives
\begin{align*}
    (z-m)^T\tilde f(z) &= e_i^T QXP A x_i = \tr[(e_ie_i^TXA)^T QXP],
    \\ \tr[\nabla \tilde f(z)] &=
\sum_{lj} \frac{\partial}{\partial x_{lj}} (e_j \otimes e_l)^T \tilde f(z)
=
\sum_{lj} \frac{\partial}{\partial x_{lj}} 
\Bigl(e_l^TQ e_i e_i^T X A P e_j\Bigr).
\end{align*}
In the second line, only $XA$ depends on $x_{lj}$ and using
$\partial B^{-1} = -B^{-1} (\partial B) B^{-1}$,
$$\frac{\partial}{\partial x_{lj}} (XA) = e_l e_j^T A - X A ( X^T D e_l e_j^T + e_j e_l^T D X ) A.$$
This gives
$
\tr[\Lambda \nabla f(z)]=
    \tr[A P] e_i^TQ e_i
    -
    \tr[AP] e_i^T Q D X A x_i
    -
    e_i^TQDXAPAX^T e_i
.$
With the definition
$$
\Rem_i'\coloneq
    e_i^T QXP A x_i
    -
    \tr[A P] e_i^TQ e_i
    +
    \tr[AP] e_i^T Q D X A x_i
    +
    x_i^T A P AX^T DQe_i,
$$
we get by \eqref{SOS_tilde_f}
\begin{equation*}
\tE\Bigl[\Bigl(
\Rem_i'
\Bigr)^2\Bigr]
\le
\tE\Bigl[
\|e_i^TXA\|_F^2
+
\sum_{j=1}^p
\sum_{l=1}^n
\|\frac{\partial}{\partial x_{lj}} e_i^T XA\|_F^2
\Bigr]
.
\end{equation*}
Summing over $i\in[n]$ and using the definition
of $f^i$ in the right-hand side,
$$
\sum_{i=1}^n\tE\Bigl[(\Rem_i')^2\Bigr]
\le 
\tE\Bigl[
\|XA\|_F^2
+
\sum_{j=1}^p
\sum_{l=1}^n
\|\frac{\partial}{\partial x_{lj}} 
(AX^T)\|_F^2
\Bigr].
$$
We have
$
\frac{\partial}{\partial x_{lj}} 
(AX^T)
= Ae_j e_l^T
- A(e_j e_l^T D X + X^T D e_l e_j^T)AX^T
$ by direct differentiation. Hence
using $\frac13(a+b+c)^2 \le a^2 + b^2 + c^2$,
\begin{align}
    \nonumber
\frac13
\sum_{j=1}^p
\sum_{l=1}^n
\|\frac{\partial(AX^T)}{\partial x_{lj}} 
\|_F^2
&\le 
\sum_{j=1}^p
\sum_{l=1}^n
\|Ae_j\|_2^2
+
\|Ae_j\|_2^2
\|e_l^TDXAX^T\|_2^2
+ \|AX^TDe_l\|_2^2 \|e_j^TAX^T\|_2^2
\\&
    \nonumber
= n\|A\|_F^2
+ \|A\|_F^2 \|DXAX^T\|_F^2
+ \|AX^TD\|_F^2 \|AX^T\|_F^2
\\&
    \nonumber
\le  n\|A\|_F^2
+ \|A\|_F^2 \|A^{1/2}X^T\|_F^2
+ \|A^{1/2}\|_F^2 \|AX^T\|_F^2
\\&\le  n\|A\|_F^2
+ \|A\|_F^2 \|A\|_{op} \|X^T\|_F^2
+ \tr[A] \|A\|_{op}^2 \|X^T\|_F^2
\label{rhs_37}
\end{align}
thanks to $\|D^{1/2}XA^{1/2}\|_{op}\le 1$ for the last inequality.
We have
    \begin{equation}
        \label{bounds_op_norm}
        \begin{aligned}
    \|A\|_{op}
    &
    \le \|H^{-1}\|_{op}\le1/({\mu n}),
    \\
    \|D^{1/2}XAX^TD^{1/2}\|_{op}
    &
    \le 1,
    \\
    \|D^{1/2}XA\|_{op}
    &
    \le\|H^{-1}\|_{op}^{1/2}
    \le 1/\sqrt{\mu n}
    \end{aligned}
    \end{equation}
by definition of $A$ and the fact that $H - n\mu I_p$ is positive semi-definite.
Thus $\|A\|_F^2 \le p /(n\mu)^2$
and $\tr[A] \le p/(n\mu)$, and
the expectation of the right-hand side \eqref{rhs_37} is bounded from above
by the right-hand side of \eqref{conclusion33}.

It remains to bound from above $\sum_{i=1}^n\tE[(\Rem_i - \Rem_i')^2]$,
as $\Rem_i'$ is not exactly $\Rem_i$ in \eqref{Rem_i} due
to the presence of the orthogonal projectors $P$ and $Q$
the additional term $x_i^TAPAX^TDQe_i$ in $\Rem_i'$ that has
no counterpart in $\Rem_i$.
Next, 
    \begin{align*}
    \Rem_i - \Rem_i'
    &=
    e_i^T(X - QXP)A x_i
    - x_i^T A PAX^TD Q e_i
  \\&\quad\quad
    +
    e_i^T\bigl(\tr[A P] Q  - \tr[A] I_n\bigr)e_i
    +
    e_i^T\bigl(\tr[A] - Q\tr[AP]\bigr) DXAx_i
  \\&\coloneq A_i + B_i + C_i + E_i.
    \end{align*}
    We now bound the four terms on the right-hand side
    using linear algebra. Since
    the sum of squares of diagonal elements of a matrix
    is bounded by its squared Frobenius norm, we have
    with $S=
    \sum_{i=1}^n
    A_i^2+B_i^2+C_i^2+E_i^2$ that
    \begin{align*}
    S
    &\le
    \|(X-QXP)AX^T\|_F^2
    + \|XAPAX^TDQ\|_F^2
    \\
    &\qquad +\|\tr[AP]Q - \tr[A] I_n\|_F^2
    +\|(\tr[A] I_n - \tr[AP]Q)DXAX^T\|_F^2
  \\&\le
    \|(X-QXP)AX^T\|_F^2
    + \|XAPAX^TD \|_F^2
    +
    \|\tr[AP]Q - \tr[A] I_n\|_F^2
    (1 + \|DXAX^T\|_{op}^2).
    \end{align*}
    thanks to $\|MM'\|_F\le\|M\|_F\|M'\|_{op}$ that we will use
    repeatedly.
    Using $\|DXA^{1/2}\|_{op}\le 1$ from \eqref{bounds_op_norm}
    where possible,
    \begin{align}
        \label{upper_bound_S}
        S&\le
    \|X-QXP\|_F^2\|XA\|_{op}^2
    +\|XA\|_{op}^2\|A^{1/2}\|_{F}^2
    +\|\tr[AP]Q - \tr[A] I_n\|_F^2
    (1+
    \|A^{1/2}X^T\|_{op}^2
    ).
    \end{align}
    Since $X-QXP = (I_n-Q)X +QX(I_p - P)$,
    the rank of $X-QXP$ is at most $d_P + d_Q$ hence
    \begin{align*}
        \|X-QXP\|_F^2
        &\le (d_P + d_Q) \|X-QXP\|_{op}^2
      \\&\le (d_P + d_Q) (2\|X\|_{op}^2 + 2\|QXP\|_{op}^2)
    \le 4(d_P + d_Q)\|X\|_{op}^2.
    \end{align*}
    For the last term in the upper bound on $S$,
    \begin{align*}
    \tr[A] I_n - Q \tr[AP]
    &= \tr[A](I_n - Q) + Q(\tr[A] - \tr[AP]),
    \\
    \|\tr[A] I_n - Q \tr[AP]\|_F
      &\le
     \|\tr[A](I_n - Q)\|_F +
     \|Q(\tr[A] - \tr[AP])\|_F.
    \end{align*}
    For the second term, simply use $\|Q\|_F\le \sqrt n$
    so that
    $\|Q(\tr[A] - \tr[AP])\|_F \le \sqrt n \tr[A(I_p - P)]
    \le\sqrt n \|A\|_{op} d_P$.
    The first term on the other hand
    is $\tr[A] \sqrt{d_Q}$.
    We have thus established the following upper bound on $S$:
    $$
    S \le 
    \Bigl(4(d_P+d_Q)\|X\|_{op}^2 + \|A^{1/2}\|_F^2\Bigr)\|XA\|_{op}^2
    + \Bigl(
        2n \|A\|_{op}^2 d_P^2 
        +  2 \tr[A]^2 d_Q
    \Bigr)(1+\|A^{1/2}X^T\|_{op}^2).
    $$
    We use $\|A\|_{op}\le (n\mu)^{-1}$ and
    $\tr[A]\le p/(n\mu)$ to get
    that $\tE[S]$ is bounded from above by the
    right-hand side of \eqref{conclusion33}.
    Since
    $\sum_{i=1}^n\tE[(\Rem_i-\Rem_i')^2] \le 4 \tE[S]$,
    this completes the proof.
\end{proof}

\begin{corollary}[A tilted affine Gaussian law]
    \label{cor:tilted-proba}
    Let $g_0$ be the affine Gaussian law of $X$ in \Cref{thm:proba},
    with all the matrices in that theorem deterministic.
    Parametrize its support by orthonormal free coordinates $z\in\R^d$
    after whitening $X$ to $X\Sigma^{-1/2}$, so that $z$ is standard
    Gaussian under $g_0$. Let $\pi$ have density proportional to
    $w(z)>0$ with respect to $g_0$, where $\log w$ is locally Lipschitz
    and $\|\nabla\log w\|_2\le K$ almost everywhere.
    Then
    \begin{equation}
        \label{conclusion33-tilted}
        \sum_{i=1}^n\E_\pi[\Rem_i^2]
        \le C(\mu,\delta)(1+K^2)(1+d_P^2+d_Q)
        \E_\pi\Bigl[
        \Bigl(1+\frac{\|X\Sigma^{-1/2}\|_{op}}{\sqrt n}\Bigr)^4
        \Bigr].
    \end{equation}
    This also holds for the support-restricted matrix $A$ in
    \eqref{Alim}, with a deterministic orthogonal projection $W$.
    The deterministic denominator bound \eqref{weight_bounded_from_1}
    and \eqref{conclusion_312}, with every expectation taken under
    $\pi$, hold as well.
\end{corollary}
\begin{proof}
    Whiten as in the proof of \Cref{thm:proba}.
    Apply \Cref{lemma:weighted-second-stein} in the free coordinates
    to the projection of $\vect(e_ie_i^TXAP)$ onto those coordinates.
    Its Gaussian divergence discrepancy is precisely $\Rem_i'$ in
    that proof. Orthogonal projections do not increase either of the
    two norms on the right-hand side of the weighted Stein inequality.
    Thus the displayed bounds on $\E[(\Rem_i')^2]$ and their sum in
    that proof hold under $\pi$ after multiplication of their
    right-hand sides by $\max\{2+3K^2/2,4\}\le4(1+K^2)$.
    All remaining estimates in that proof are pointwise linear
    algebra inequalities, so they can be integrated under $\pi$.
    This proves \eqref{conclusion33-tilted}.

    For \eqref{Alim}, apply this result to
    $A_t=(X^TDX+\Sigma^{1/2}H\Sigma^{1/2}+tW)^{-1}$ and let
    $t\to\infty$, keeping the law $\pi$ and its weight fixed.
    Fatou's lemma suffices for the second-moment inequality.
    The proof of \eqref{weight_bounded_from_1} is pointwise and passes
    to the same limit, and then gives \eqref{conclusion_312} under
    $\pi$ without any integration by parts involving $I_F$.
\end{proof}

To obtain $\tr[\Sigma\hat A]\approx \frac{x_i^TAx_i}{1-D_{ii}x_i^TAx_i}$ from
the bound \eqref{conclusion33} on $\Rem_i$,
the last ingredient we will need is that $1-D_{ii}x_i^TAx_i$ in
the definition of $\Rem_i$ in \eqref{Rem_i} is bounded away from 0
since whenever $1-D_{ii}x_i^TAx_i \ne 0$ it holds
$$
\Rem_i 
(1-D_{ii}x_i^TAx_i)^{-1}
= \frac{x_i^T A x_i}{1-D_{ii}x_i^TAx_i} - \tr[\Sigma A] .
$$

\begin{proposition}
    Let the setting and assumptions of \Cref{thm:proba} be fulfilled.
    Define the event
    \begin{equation}
        \label{def_event_F}
        F=\cap_{i=1}^n \{ \|\Sigma^{-1/2}x_i\|_2 \le  \sqrt p + \sqrt n\}
        .
    \end{equation}
    Then, with $I_F$ the indicator of $F$,
    \begin{equation}
    \E\Bigl[
    I_F
    \sum_{i=1}^n
    \Big|\tr[\Sigma A] - \frac{x_i^T A x_i}{1-D_{ii}x_i^TAx_i}\Big|
    ^2
    \Bigr]
    \le
    \Bigl(1 + \frac{(\sqrt p + \sqrt n)^2}{(n\mu)}\Bigr)^2
    \sum_{i=1}^n \E\Bigl[\Rem_i^2\Bigr]
    \label{conclusion_312}
    \end{equation}
\end{proposition}
\begin{proof}
To prove that $1-D_{ii}x_i^TAx_i$ is bounded from below,
with the notation and assumptions
in \Cref{thm:proba} we have thanks to $n\mu I_p \preceq H$ (which implies $\Sigma^{1/2}H\Sigma^{1/2}\succeq n\mu\Sigma$)
and $X^TDX \succeq x_iD_{ii} x_i^T$ that
\begin{equation}
    \label{above}
x_i^T A x_i
\le x_i^T (x_i D_{ii} x_i^T + n\mu\Sigma)^{-1} x_i
=
\frac{x_i^T\Sigma^{-1} x_i}{D_{ii}x_i^T\Sigma^{-1} x_i + n\mu}
\end{equation}
by comparing positive definite matrices
inside the inverse,
and solving the linear system $(x_i D_{ii} x_i^T + n\mu \Sigma)^{-1}x_i$ using
the identity
$$
(x_iD_{ii} x_i^T + n\mu\Sigma) \frac{\Sigma^{-1}x_i}{D_{ii} x_i^T\Sigma^{-1}x_i +  n\mu} = x_i.
$$
Multiplying by $-D_{ii}$ and adding 1,
inequality \eqref{above} is equivalent to
$$
1-D_{ii}x_i^TAx_i
\ge
1 - \frac{D_{ii}x_i^T\Sigma^{-1} x_i}{D_{ii}x_i^T\Sigma^{-1} x_i + n\mu}
= \frac{n\mu}{D_{ii}x_i^T\Sigma^{-1} x_i + n\mu},
$$
so that $1-D_{ii}x_i^TAx_i$ is bounded away from 0 in the sense
\begin{equation}
0 < (1 - D_{ii}x_i^TAx_i)^{-1}
\le 1 + D_{ii}x_i^T\Sigma^{-1}x_i/(n\mu).
\label{weight_bounded_from_1}
\end{equation}
On the event $F$, the right-hand side is at most
$1+(\sqrt p+\sqrt n)^2/(n\mu)$ since $D_{ii}\in[0,1]$.
Combining this bound with the identity preceding the proposition gives
\eqref{conclusion_312}.
\end{proof}

In the original iid Gaussian design before disintegration,
$X$ has iid $N(0,\Sigma)$ rows so that
$F$ has probability at least $1-ne^{-n/2}$. Indeed
$x_i^T\Sigma^{-1}x_i$ has $\chi^2(p)$ distribution, so that
$\mathbb P(\|\Sigma^{-1/2}x_i\|_2 \le \sqrt p + t)
\ge 1-e^{-t^2/2}$ for all $t\ge 0$
\cite[Lemma 1]{laurent2000adaptive}, and the claim follows from the
union bound.

\subsection*{Restriction to a support or a subspace}

In order to apply the above result to \Cref{example:enet,example:group},
it is useful to state the following corollary, where
\begin{equation}
A = \lim_{t\to+\infty}(X^TDX + \Sigma^{1/2}H\Sigma^{1/2} + t W)^{-1}
\label{Alim}
\end{equation}
for some orthogonal projection $W\in\R^{p\times p}$.
If $S\subset [p]$ is a fixed
subset of indices and $W=\sum_{j\in S^c}e_j e_j^T$,
the above limit is given by
\begin{equation}
A_{S,S} = (X_S^T D X_S + (\Sigma^{1/2}H\Sigma^{1/2})_{S,S})^{-1},
\qquad A_{jk} = 0 \text{ if } j\notin S \text{ or } k\notin S,
\label{A_S_S}
\end{equation}
for instance by applying the block inversion formula before 
taking $t\to +\infty$.

    \begin{corollary}
    \label{cor:support}
    Let the assumptions of \Cref{thm:proba} be fulfilled.
    Let $W$ be a deterministic orthogonal projection.
    Replace the definition of $A$ by \eqref{Alim}
    in \Cref{thm:proba}.
    Then the conclusions
    \eqref{conclusion33},
    \eqref{weight_bounded_from_1}
    and \eqref{conclusion_312}
    still hold.
\end{corollary}
With $A_t = (X^TDX + \Sigma^{1/2}H\Sigma^{1/2} + t W)^{-1}$
for $t\ge 0$, we have
the monotone limits
$\tr[\Sigma^{1/2} A_t \Sigma^{1/2}]\downarrow \tr[\Sigma^{1/2}A\Sigma^{1/2}]$ and
$x_i^T A_t x_i \downarrow x_i^T A x_i$
as $t\uparrow+\infty$.
Applying \Cref{thm:proba} to $A_t$ and taking $t\to+\infty$,
using the dominated convergence theorem gives \Cref{cor:support}.

\subsection*{What is the catch?}
The notation and setting
used so far in \Cref{sec:proba}
is seemingly detached from the notation in
\Cref{example:enet,example:group} or \Cref{prop:mean_field}.
In the result \Cref{thm:proba}, the matrices
$D$ and $H$ are deterministic, and the support 
$S$ in \eqref{A_S_S} is also deterministic
for \Cref{cor:support} to hold with the deterministic
projection $W=\sum_{j\in S^c}e_j e_j^T$.

In the setting discussed in the introduction,
the corresponding quantities are random.
With \Cref{example:twice}, we are interested in
$$
\hat A = (X^TD X  + \nabla^2 R(\hat b))^{-1}
$$
which is random because $\nabla^2 R(\hat b)$ is random through
$\hat b$, and $D=\diag(\{L_{y_i}''(x_i^T\hat b), i\in[n]\})$ is also
random through $L_{y_i}''(x_i^T\hat b), i\in[n]$.
In \Cref{example:enet,example:group}, the support $\hat S=
\{j\in[p]:\hat b_j\ne 0\}$ is random as well,
and it is unclear at this point why \Cref{thm:proba}
and \Cref{cor:support} with a deterministic $S\subset[p]$ in
\eqref{A_S_S} would be useful at all.
This will be resolved by disintegrating over particular affine fibers.
On each fiber, constraints of the form
$X(I_p-P)=\bar U$ and $(I_n-Q)X=\bar V$ make
$D,\hat b, \nabla^2 R(\hat b)$, and in the case
of \Cref{example:enet,example:group}
the support $\hat S$,
are fully determined.  Such fibers have been studied in
\cite{celentano2021cad,celentano2023challenges}
and are described in the following sections.

\section{Affine fiber example I: Robust linear regression}
\label{sec:4}
Before we move to generalized linear models and 
single-index models in \Cref{sec:glm},
let us apply the weighted result
in \Cref{cor:tilted-proba} to
robust linear regression, where a linear model
\begin{equation}
    \label{linear_model}
y_i = x_i^T\beta^* + \eps_i
\end{equation}
is assumed, with $x_i\sim N(0, \Sigma)$ and $\eps_i$ is
independent of $X$.
Consider throughout this section a convex $C^2$ loss function
$\rho:\R\to\R$ with $0\le\rho''\le1$, so that the minimization problem
is
\begin{equation}
\label{hat_b_robust}
\hat b = \argmin_{b\in \R^p}
\sum_{i=1}^n \rho(y_i - x_i^T b )
+ R(b)
\end{equation}
The argument in this section does not cover the nonsmooth second
derivative of the Huber loss. The penalties may nevertheless be
nonsmooth, as in \Cref{example:enet,example:group}.

The key intuition behind the argument of \Cref{prop:square_loss}
is that in the square loss case, if
$Q\Lambda P^T$ is the SVD of the design matrix $[X\mid \eps]\in\R^{n\times (p+1)}$, then
$\hat b$ and $\tr[H]$ are independent of the left singular vectors $Q$.
We can then argue conditionally on $(\Lambda,P)$ with respect to
the randomness of the independent matrix $Q$.

This argument fails as soon as the loss $L_{y_i}(\cdot)$
is not the square loss: the rotational invariance 
argued in \eqref{argue_rotational_inv} is lost.
Instead, we disintegrate over affine fibers of the Gaussian design.
We first fix the noise $\eps$ (which is independent of $X$).
The only remaining randomness comes from the design matrix $X$
that we assume has iid $N(0,\Sigma)$ rows.
Consider deterministic $\bar h\in\R^p,  \bar u\in\R^n,
\bar v \in\R^n,
\bar w\in\R^p$ with
\begin{equation}
\bar w\in\partial R(\bar h + \beta^*),
\qquad
\bar v_i = \rho'(\eps_i - \bar u_i).
\label{constraint_v_u_h_w}
\end{equation}
Then consider the affine fiber
\begin{equation}
X\bar h = \bar u,
\quad
X^T\bar v = \bar w.
\label{michael_event}
\end{equation}
Such affine fibers are studied in \cite{celentano2021cad,celentano2023challenges}
in order to apply the Convex Gaussian Minmax Theorem of
\cite{thrampoulidis2018precise} sequentially.
The first key observation made in \cite{celentano2021cad,celentano2023challenges} is that on the fiber \eqref{michael_event},
the vector $\bar h + \beta^*$ is solution to the optimization problem
\eqref{hat_b_robust} because the KKT conditions
$$
X^T\bar v
=
\sum_{i=1}^n x_i \rho'\Bigl(y_i - x_i^T(\bar h + \beta^*)\Bigr) \in \partial R(\bar h + \beta^*)
$$
are satisfied. Hence on this fiber
$\hat b = \bar h + \beta^*$, $X(\hat b - \beta^*) = \bar u$
and $\bar w$ is an element of $\partial R(\hat b)$ that satisfies
the KKT conditions for $\hat b$. 
The fiber \eqref{michael_event} is defined by linear constraints on
$X$, but this alone does not identify the regular conditional law
given the fitted quantities. The relevant conditioning map is
nonlinear, and its normal Jacobian must be included.
\Cref{lemma:regression-coarea} proves that, on the nondegenerate
fibers, this law has density proportional to
\[
    \det\bigl(X_{\hat S}^TDX_{\hat S}+B_{\hat S}\bigr)
\]
relative to the affine Gaussian law on \eqref{michael_event}.
Here $\hat S=[p]$ for a smooth penalty, or the active coordinate set
(the union of active groups) for the other two examples, and
$B_{\hat S}$ is the Hessian of the penalty restricted to those
coordinates. Both $B_{\hat S}$ and $D$ are fixed on the fiber.
Conditioning also on the inactive components of $X^T\rho'(y-X\hat b)$
ensures that there are no remaining selection inequalities restricting
the design within this affine space.

The log-determinant tilt has gradient of norm at most $1/\sqrt\mu$
in the whitened free Gaussian coordinates. We can therefore use
\Cref{cor:tilted-proba} under the actual conditional law, and then
average its second-moment bounds. The shift $h=\hat b-\beta^*$ is
important: the residual is $\eps-Xh$, so just one right-hand
direction, $h$, suffices. Neither $X\beta^*$ nor the responses $y$
need to be fixed separately.

\begin{corollary}
    \label{cor_robust}
    Let $\mu,\delta>0$ 
    be constants. Assume the proportional regime \eqref{proportional_regime}.
    Consider the robust regression setting 
    \eqref{linear_model}-\eqref{hat_b_robust}
    with $\eps$ independent of $X$
    and loss $\rho\in C^2(\R)$ convex with $0\le\rho''\le1$.
    Assume that $X$ has iid $N(0,\Sigma)$ rows
    with $\Sigma$ symmetric positive definite.
    Consider one of the regularizers
    in \Cref{example:twice,example:enet,example:group}
    and assume $\nu \ge \mu \|\Sigma\|_{op}$ 
    so that the strong convexity \eqref{mu-n-strongly-convex} holds.
    In \Cref{example:twice}, assume additionally that $R\in C^2$.
    Elastic-net with $\lambda=0$ is treated as the smooth ridge
    case, using the matrix in \Cref{example:twice}.
    For the explicit matrix $\hat A$ given 
    in \Cref{example:twice,example:enet,example:group} we have
    \begin{equation}
    \E\Bigl[I_F
    \sum_{i=1}^n
    \Big|
    \frac{x_i^T \hat A x_i}{1-D_{ii}x_i^T \hat A x_i}
    - \tr[\hat A\Sigma]
    \Big|^2
    \Bigr]
    =
    \E\Bigl[I_F
    \sum_{i=1}^n
    \Big|
    W_i
    - \tr[\hat A\Sigma]
    \Big|^2
    \Bigr]
    \le C(\mu,\delta)
    \label{robust_regression_weight}
    \end{equation}
    where $F$ is the event \eqref{def_event_F} of
    exponentially high probability and $I_F$ its indicator function.
    If additionally the regularizer
    $R(\cdot)$ is minimized at 0, the loss
    $\rho$ is 1-Lipschitz and the test function
    $g$ satisfies
    $\sup_{y\in\mathcal Y}|g(x,y) - g(x',y)|\le |x-x'|(1+|x|+|x'|)$
    then,
    for the leave-one-out estimate $\loo$ defined in \eqref{loo},
    \begin{equation}
    \E\Bigl[I_F
    \Big|
    \loo
    -
    \frac1n
    \sum_{i=1}^n
    g\Bigl(
        x_i^T\hat b + \frac{x_i^T\hat Ax_i}{1-D_{ii}x_i^T\hat Ax_i}
        L_{y_i}'(x_i^T\hat b),
    ~
    y_i
    \Bigr)
    \Big|
    \Bigr] 
    \le \frac{C(\mu,\delta)}{\sqrt n}.
    \label{robst_conclusion_LOO}
    \end{equation}
\end{corollary}

The display \eqref{robst_conclusion_LOO} can be rewritten
$\E[I_F|\loo - \alo|] \le C(\mu,\delta)/\sqrt n$.
Since the event $F$ in \eqref{def_event_F} has exponentially
large probability in $n$, this proves
that $|\loo - \alo| = O_P(n^{-1/2})$.
If, in addition, the assumptions of \cite{zou2024theoretical} give
\[
    |\loo-\Err(\hat b)|=O_P(n^{-1/2}),
\]
combining the two bounds establishes the consistency
of the ALO-CV estimates \eqref{alo} for estimating the generalization
error \eqref{err_b_hat}.

\begin{proof}[Proof of \Cref{cor_robust}]
    Fix $\eps$, independently of $X$, and disintegrate with respect to
    \[
        h=\hat b-\beta^*,\qquad u=Xh,\qquad
        q=X^T\rho'(\eps-u).
    \]
    Write $v=\rho'(\eps-u)$, and let $\pi$ denote the actual
    conditional law. The fitted vector, $D$, the active set $S$ and
    the active penalty Hessian $B_S$ are all fixed.
    By \Cref{lemma:regression-coarea}, for $S\ne\varnothing$ and
    $h,v\ne0$, $\pi$ is the determinant tilt of the affine Gaussian
    law on $\{X:Xh=u,\ X^Tv=q\}$, with $K^2\le1/\mu$.

    In whitened design coordinates the Gaussian part has free
    perturbations $QZ\widetilde P$, where
    \[
        \widetilde P=I_p-
        \frac{\Sigma^{1/2}hh^T\Sigma^{1/2}}{h^T\Sigma h},
        \qquad Q=I_n-\frac{vv^T}{\|v\|_2^2}.
    \]
    Thus take $P=\Sigma^{-1/2}\widetilde P\Sigma^{1/2}$ in
    \Cref{cor:tilted-proba}; both kernel dimensions equal one.
    In the smooth case take
    $H=\Sigma^{-1/2}\nabla^2R(\hat b)\Sigma^{-1/2}$.
    For either nonsmooth penalty, extend $B_S$ to the block diagonal
    matrix $B=\operatorname{blockdiag}(B_S,n\nu I_{S^c})$ in the
    coordinate decomposition $S\cup S^c$, and take
    $H=\Sigma^{-1/2}B\Sigma^{-1/2}$ and
    $W=\sum_{j\in S^c}e_je_j^T$ in \eqref{Alim}.
    This gives precisely the matrix $\hat A$ in the corresponding
    example, and $H\succeq n\mu I_p$ in all three cases.

    The same lemma handles the exceptional cases: an empty active
    set gives zero error; $h=0,v\ne0$ with nonempty $S$ has probability
    zero; and when $v=0$ the conditional law is affine Gaussian, so
    use $K=0$, $Q=I_n$ and at most one right-hand direction.
    A zero right-hand direction is simply omitted.
    Combining \eqref{conclusion33-tilted} with the pointwise
    denominator bound on $F$ consequently gives, for almost every
    conditioning value,
    \[
        \E_\pi\Bigl[I_F\sum_i
        |W_i-\tr[\Sigma\hat A]|^2\Bigr]
        \le C(\mu,\delta)\E_\pi\Bigl[
        \Bigl(1+\frac{\|X\Sigma^{-1/2}\|_{op}}{\sqrt n}\Bigr)^4\Bigr].
    \]
    Average first over the fitted quantities and then over $\eps$.
    Because both sides use the actual conditional law, the tower
    property leaves on the right only the original iid Gaussian
    design moment, which is bounded by $C(\delta)$ by Gaussian
    operator-norm bounds
    \cite[Theorem 2.13]{DavidsonS01}.
    This proves \eqref{robust_regression_weight} without a
    fourth-moment transfer or a cost depending on the number of supports.

    For the additional assertion, write
    \[
        a_i=x_i^T\hat b+W_iL_{y_i}'(x_i^T\hat b),\qquad
        c_i=x_i^T\hat b+\tr[\Sigma\hat A]L_{y_i}'(x_i^T\hat b).
    \]
    The assumption on $g$ and Cauchy--Schwarz, on the product of the
    probability space and the uniform measure on $[n]$, give
    \begin{align*}
        \E\Bigl[I_F\Big|\frac1n\sum_i
        \bigl(g(a_i,y_i)-g(c_i,y_i)\bigr)\Big|\Bigr]
        &\le
        \Bigl(\E\Bigl[\frac{I_F}{n}\sum_i|a_i-c_i|^2\Bigr]\Bigr)^{1/2}\\
        &\quad{}\times
        \Bigl(\E\Bigl[\frac{I_F}{n}\sum_i
        (1+|a_i|+|c_i|)^2\Bigr]\Bigr)^{1/2}.
    \end{align*}
    The first factor is at most $C(\mu,\delta)/\sqrt n$ by
    \eqref{robust_regression_weight} and $|L_{y_i}'|\le1$.
    The second is bounded by $C(\mu,\delta)$: use
    $\tr[\Sigma\hat A]\le p/(n\mu)$,
    \eqref{robust_regression_weight} to bound
    $\E[I_F n^{-1}\sum_iW_i^2]$, and
    \eqref{bound_b_hat} with Gaussian operator-norm moments to obtain
    $\E[\|X\hat b\|_2^2/n]\le C(\mu,\delta)$.

    Finally, combine this bound with \eqref{bound_mean_field_LOO}
    and the triangle inequality to prove \eqref{robst_conclusion_LOO}.
    The use of the active-set matrices in that mean-field bound,
    including possible inactive KKT equalities, is justified in
    \Cref{remark:regression-active-derivatives}.
\end{proof}

\section{Affine fiber example II: Single-index models}
\label{sec:glm}

In this section, consider a single index model for $y_i\mid x_i$,
of the form
\begin{equation}
    \label{single_index}
    y_i = F(x_i^Tw, \eps_i)
\end{equation}
where $\eps_i$ is some external randomness independent of $X$,
and $w\in\R^p$ is an unknown deterministic index, normalized
with $\E[(x_i^Tw)^2]=1$ by convention (as the amplitude
of $w$ can be otherwise absorbed into $F$).
In binary classification, we may for instance take $\eps_i\sim$ Unif$[0,1]$
and $y_i = I\{ \eps_i \le \phi(x_i^Tw) \}$
where $\phi:\R\to[0,1]$ is the sigmoid (logistic regression)
or the Gaussian CDF (probit regression).

In robust linear regression with
\eqref{linear_model}-\eqref{hat_b_robust},
there is only one direction $\bar h\in\R^p$ that we need
to fix on the right of $X$ in \eqref{michael_event},
because on that fiber (and for fixed $\eps$),
$\hat b-\beta^*$, $X(\hat b-\beta^*)$, the residuals, their scores
and their loss curvatures are all deterministic. The responses $y$
and the functions $L_{y_i}$ need not themselves be fixed.
In single index model, the nonlinearity requires us to fix
another direction on the right of $X$.

By disintegrating with respect to the external randomness $\eps_i$,
we may first fix it and argue only with respect to the randomness of $X$.
Consider,
for deterministic
$\bar u, \bar u'\in\R^n,
\bar v \in\R^n,
\bar b,  \bar h\in \R^p,
\bar y\in \mathcal Y^n
$ such that
$$
\bar y_i = F(\bar u'_i, \eps_i),
\qquad
\bar v_i = - L_{\bar  y_i}'(\bar u_i),
\qquad
\bar h
\in \partial R(\bar b)
$$
and the affine fiber
\begin{equation}
\label{michael_event_single}
\mathcal F
=\{
X\bar b = \bar u,
\quad
X w = \bar u',
\quad
X^T\bar v = \bar h
\}.
\end{equation}
On this fiber, we have by construction that $y_i=\bar y_i$
and that $\bar b$ satisfies the KKT conditions of the problem
\eqref{hat_b}, hence $\hat b = \bar b$ and
$x_i^T\hat b = \bar u_i$.
On this fiber, the responses $y_i$, the loss $L_{y_i}(\cdot)$,
the estimator $\hat b$ and its predicted values $x_i^T\hat b$,
the matrix $D=\diag(\{L_{y_i}''(x_i^T\hat b), i\in[n]\})$
and the active coordinate set $\hat S$ (the union of active groups
for the group penalty)
are all fully determined and deterministic.

As in regression, the affine form of the fiber does not imply that
the actual conditional law is Gaussian.
\Cref{lemma:single-index-coarea} computes the law under the augmented
conditioning that includes $Xw$.
The proof first conditions on $a=Xw$ and the independent external
randomness, making $y=F(a,\eps)$ fixed, and then applies coarea on
the affine slice $\{X:Xw=a\}$. This is a disintegration by the joint
conditioning map, but does not require differentiating the response
link $F$. In particular, binary response links are allowed.

On the regular fibers, the conditional density relative to the
affine Gaussian law is again proportional to
$\det(X_{\hat S}^TDX_{\hat S}+B_{\hat S})$.
The two normal directions in the restricted coarea calculation
contribute only factors that are constant along the fiber.
The log-tilt gradient therefore still satisfies $K^2\le1/\mu$,
with no lower bound on the angle between $\hat b$ and $w$.
Zero scores and collinear fitted vectors are treated separately in
the lemma.

\begin{corollary}
    \label{cor_single}
    Let $\mu,\delta>0$ 
    be constants. Assume the proportional regime \eqref{proportional_regime}.
    Consider the single index model
    \eqref{single_index}
    with $\eps$ independent of $X$,
    and the estimator $\hat b$ in
    \eqref{hat_b}.
    Assume that $X$ has iid $N(0,\Sigma)$ rows
    with $\Sigma$ symmetric positive definite.
    Assume that for all $y\in\mathcal Y$,
    the function $L_y(\cdot)$ is convex and $C^2$,
    with $0\le L_y''\le1$.
    Consider one of the regularizers
    in \Cref{example:twice,example:enet,example:group}
    and assume $\nu \ge \mu \|\Sigma\|_{op}$ 
    so that the strong convexity \eqref{mu-n-strongly-convex} holds.
    In \Cref{example:twice}, assume additionally that $R\in C^2$.
    Elastic-net with $\lambda=0$ is treated as the smooth ridge
    case, using the matrix in \Cref{example:twice}.
    For the explicit matrix $\hat A$ given 
    in \Cref{example:twice,example:enet,example:group} we have
    for the event $F$ defined in \eqref{def_event_F} that
    \begin{equation}
    \E\Bigl[
        I_F
    \sum_{i=1}^n
    \Big|
    W_i
    - \tr[\hat A\Sigma]
    \Big|^2
    \Bigr]
    =
    \E\Bigl[
        I_F
    \sum_{i=1}^n
    \Big|
    \frac{x_i^T \hat A x_i}{1-D_{ii}x_i^T \hat A x_i}
    - \tr[\hat A\Sigma]
    \Big|^2
    \Bigr]
    \le C(\mu,\delta)
    .
    \label{single_regression_weight}
    \end{equation}
    If additionally the regularizer
    $R(\cdot)$ is minimized at 0, the loss
    $L_y:\R\to\R$ is 1-Lipschitz for all $y\in\mathcal Y$
    and the test function
    $g$ satisfies
    $\sup_{y\in\mathcal Y}|g(x,y) - g(x',y)|\le |x-x'|(1+|x|+|x'|)$
    then
    with $\loo$ defined in \eqref{loo}, it holds
    \begin{equation}
    \E\Bigl[ I_F 
    \Big|
    \loo
    -
    \frac1n
    \sum_{i=1}^n
    g\Bigl(
        x_i^T\hat b + \frac{x_i^T\hat Ax_i}{1-D_{ii}x_i^T\hat Ax_i}
        L_{y_i}'(x_i^T\hat b),
    ~
    y_i
    \Bigr)
    \Big|
    \Bigr] 
    \le \frac{C(\mu,\delta)}{\sqrt n}.
    \label{single_conclusion_LOO}
    \end{equation}
\end{corollary}
\begin{proof}[Proof of \Cref{cor_single}]
    Fix $\eps$ and disintegrate with respect to
    \[
        a=Xw,\qquad b=\hat b,\qquad u=X\hat b,\qquad
        q=-X^T(L_{y_i}'(u_i))_{i\in[n]}.
    \]
    Write $v=(-L_{y_i}'(u_i))_i$ and let $\pi$ be the actual
    conditional law. The responses, the active set $S$, $D$ and
    $B_S$ are fixed. For $v\ne0$ and $b,w$ linearly independent,
    \Cref{lemma:single-index-coarea} identifies $\pi$ as the
    determinant tilt on the entire affine fiber
    $\{X:Xw=a,\ Xb=u,\ X^Tv=q\}$, with $K^2\le1/\mu$.

    Let $\widetilde P$ be the orthogonal projection onto
    $\operatorname{span}\{\Sigma^{1/2}w,\Sigma^{1/2}b\}^{\perp}$,
    and let $Q$ be the orthogonal projection onto $v^\perp$,
    taking $Q=I_n$ when $v=0$.
    The reference Gaussian law in whitened design coordinates has
    free perturbations $QZ\widetilde P$.
    Set $P=\Sigma^{-1/2}\widetilde P\Sigma^{1/2}$, and choose $H$
    and, in the nonsmooth cases, the support projection $W$ exactly
    as in the proof of \Cref{cor_robust}.
    Then $H\succeq n\mu I_p$, $d_P\le2$, $d_Q\le1$ and the matrix
    in \Cref{cor:tilted-proba} is precisely $\hat A$.

    The exceptional cases are also covered by
    \Cref{lemma:single-index-coarea}: if $v=0$ use the Gaussian
    law with $K=0$; if $v\ne0$ and $b,w$ are collinear, a nonempty
    active set has size one almost surely, and again the
    conditional law is Gaussian. Dependent or zero right-hand
    directions are omitted. An empty active set gives zero error.
    Hence \eqref{conclusion33-tilted} and the pointwise denominator
    bound on $F$ give, for almost every conditioning value,
    \[
        \E_\pi\Bigl[I_F\sum_i
        |W_i-\tr[\Sigma\hat A]|^2\Bigr]
        \le C(\mu,\delta)\E_\pi\Bigl[
        \Bigl(1+\frac{\|X\Sigma^{-1/2}\|_{op}}{\sqrt n}\Bigr)^4\Bigr].
    \]
    Averaging over $(a,b,u,q,\eps)$ uses the tower property under
    the true conditional laws. The remaining moment is under the
    original iid Gaussian design and is bounded by $C(\delta)$.
    This proves \eqref{single_regression_weight}.

    For \eqref{single_conclusion_LOO}, repeat the final
    Cauchy--Schwarz estimate in the proof of \Cref{cor_robust},
    now using \eqref{single_regression_weight}, $|L_y'|\le1$,
    $\tr[\Sigma\hat A]\le p/(n\mu)$ and \eqref{bound_b_hat}.
    Combine it with \eqref{bound_mean_field_LOO} and the triangle
    inequality. The projected derivative formula needed for the
    active-set matrices follows by the locality argument of
    \Cref{remark:regression-active-derivatives}, now on slices with
    $Xw$ and $\eps$ fixed. It requires no derivatives of $F$ and
    remains valid at possible inactive KKT equalities.
\end{proof}

\appendix
\section{Weighted Stein inequalities and regression disintegration}
\label{sec:weighted-regression}

\begin{lemma}[Second-order Stein inequality with a weight]
    \label{lemma:weighted-second-stein}
    Let $\gamma_d$ be standard Gaussian measure on $\R^d$ and let
    $d\pi=w\,d\gamma_d/\int w\,d\gamma_d$, where $w>0$,
    $0<\int w\,d\gamma_d<\infty$, and $\log w$ is locally Lipschitz
    with $\|\nabla\log w\|_2\le K$ almost everywhere.
    For a weakly differentiable $f:\R^d\to\R^d$ such that
    $\E_\pi[\|f\|_2^2+\|\nabla f\|_F^2]<\infty$, we have
    \begin{equation}
        \label{weighted-second-stein}
        \E_\pi[(z^Tf-\operatorname{div}f)^2]
        \le (2+3K^2/2)\E_\pi\|f\|_2^2
        +4\E_\pi\|\nabla f\|_F^2.
    \end{equation}
    In particular, all expectations in this inequality are under
    the tilted law; no fourth moment of $f$ or $\nabla f$ is required.
\end{lemma}
\begin{proof}
    Normalize $w$ to have integral one, write $s=\nabla\log w$,
    and put $g=\sqrt w f$. The product rule gives
    \[
        z^Tg-\operatorname{div}g
        =\sqrt w\bigl(z^Tf-\operatorname{div}f-s^Tf/2\bigr),
        \qquad
        \nabla g=\sqrt w(\nabla f+fs^T/2).
    \]
    The ordinary Gaussian second-order Stein inequality
    \cite{bellec_zhang2018second_stein} and
    $\|a+b\|^2\le2\|a\|^2+2\|b\|^2$ therefore imply
    \begin{align*}
        \E_\pi[(z^Tf-\operatorname{div}f)^2]
        &\le2\E_{\gamma_d}[(z^Tg-\operatorname{div}g)^2]
             +\tfrac12\E_\pi[(s^Tf)^2]\\
        &\le2\E_\pi\|f\|_2^2
             +2\E_\pi\|\nabla f+fs^T/2\|_F^2
             +\tfrac12K^2\E_\pi\|f\|_2^2\\
        &\le(2+3K^2/2)\E_\pi\|f\|_2^2
             +4\E_\pi\|\nabla f\|_F^2.
    \end{align*}
    These calculations extend from smooth compactly supported fields
    by Gaussian Sobolev approximation. Indeed, the product rule gives
    $g\in W^{1,2}(\gamma_d;\R^d)$.
    Gaussian divergence is continuous from this space into
    $L^2(\gamma_d)$ by the same second-order inequality.
\end{proof}

\begin{lemma}[A dimension-free determinant tilt]
    \label{lemma:logdet-tilt}
    Let $\mu>0$ and let $\Sigma$ be symmetric positive definite.
    Fix a nonempty coordinate set $S\subset[p]$, a symmetric matrix
    $B_S\succeq n\mu\Sigma_{S,S}$, and a diagonal matrix $0\preceq D\preceq I_n$.
    Put $\mathcal S_S=B_S+X_S^TDX_S$.
    In the whitened coordinates $X\Sigma^{-1/2}$,
    \begin{equation}
        \label{logdet-gradient-bound}
        \|\nabla\log\det\mathcal S_S\|_F^2
        \le\frac{\min\{n,|S|\}}{n\mu}\le\frac1\mu.
    \end{equation}
    The same bound holds after restriction to orthonormal coordinates
    on any affine subspace of the whitened design space.
\end{lemma}
\begin{proof}
    Let $Y=X_S\Sigma_{S,S}^{-1/2}$,
    $B=\Sigma_{S,S}^{-1/2}B_S\Sigma_{S,S}^{-1/2}\succeq aI$,
    where $a=n\mu$, and $T=B+Y^TDY$.
    Differentiation gives $G=\nabla_Y\log\det T=2DYT^{-1}$.
    Since $D^2\preceq D$,
    \[
        G^TG\preceq4T^{-1}(T-B)T^{-1}
        \preceq4(T^{-1}-aT^{-2})\preceq a^{-1}I.
    \]
    The last inequality follows from $4(t-a)/t^2\le1/a$ for $t\ge a$.
    As $\rank G\le\min\{n,|S|\}$, this bounds $\|G\|_F^2$.
    If $E_S$ is the coordinate inclusion matrix, then
    $Y=(X\Sigma^{-1/2})C_S$ with
    $C_S=\Sigma^{1/2}E_S\Sigma_{S,S}^{-1/2}$ and $C_S^TC_S=I$.
    Hence the full whitened gradient has the same Frobenius norm.
    Restricting to an affine subspace projects this gradient and
    cannot increase its norm.
\end{proof}

\begin{lemma}[The regression conditional law]
    \label{lemma:regression-coarea}
    Fix $(\beta^*,\eps)\in\R^p\times\R^n$ and let $X$ have iid
    $N(0,\Sigma)$ rows with $\Sigma$ positive definite.
    Suppose $\rho\in C^2(\R)$ is convex with $0\le\rho''\le1$.
    Let $R$ be one of the three penalties of
    \Cref{example:twice,example:enet,example:group}, with $R\in C^2$
    in the first case and $\nu\ge\mu\|\Sigma\|_{op}>0$.
    Define
    \begin{gather*}
        \hat b=\argmin_b\Bigl\{\sum_i\rho(\eps_i-x_i^T(b-\beta^*))+R(b)\Bigr\},\\
        h=\hat b-\beta^*,\qquad u=Xh,\qquad
        v=\rho'(\eps-u),\qquad q=X^Tv.
    \end{gather*}
    Here $S=[p]$ in the smooth case, including pure ridge
    (elastic-net with $\lambda=0$). For elastic-net with $\lambda>0$,
    $S$ is the support; for group-lasso it is the union of the
    active groups. For nonempty $S$, let
    \[
        B_S=\nabla^2 R_S(\hat b_S),\qquad
        D=\diag(\rho''(\eps-u)),\qquad
        \mathcal S_S=B_S+X_S^TDX_S,
    \]
    Here $R_S(b_S)=R(b_S,0_{S^c})$; in the nonsmooth examples its
    derivatives are taken where the active coordinates or groups
    are nonzero.
    For almost every realized $(h,u,q)$ with $S\ne\varnothing$ and
    $h,v\ne0$, the conditional law $\pi_{h,u,q}$ of $X$ has density
    \begin{equation}
        \label{regression-density-tilt}
        \frac{d\pi_{h,u,q}}{dg_{h,u,q}}(X)
        =\frac{\det\mathcal S_S(X)}
        {\E_{g_{h,u,q}}[\det\mathcal S_S(X)]},
        \qquad
        \mathcal F_{h,u,q}=\{X:Xh=u,\ X^Tv=q\},
    \end{equation}
    where $g_{h,u,q}$ has density proportional to
    $\exp(-\|X\Sigma^{-1/2}\|_F^2/2)$
    with respect to Hausdorff measure on $\mathcal F_{h,u,q}$.
    The quantities $S,B_S,D$ are fixed on this
    entire fiber. The log-density in \eqref{regression-density-tilt}
    has gradient of norm at most $1/\sqrt\mu$ in its whitened free
    Gaussian coordinates.

    The event $\{S\ne\varnothing,h=0,v\ne0\}$ has probability zero.
    On $\{v=0\}$ the conditional law is affine Gaussian without a
    tilt. If $S=\varnothing$ in either nonsmooth example, $\hat A=0$
    and the weight approximation error is zero.
\end{lemma}
\begin{proof}
    Strong convexity gives uniqueness. In the elastic-net case with
    $\lambda=0$ use the smooth case. Otherwise consider each fixed
    candidate active coordinate set $S$ (a union of groups for the
    group penalty), and let $b^S$ minimize the objective with
    $b_{S^c}=0$. The restricted estimator is continuous in $X$ by
    strong convexity and continuity of the objective and its loss
    gradient. The set $O_S$ where every designated active coordinate
    or group of $b^S$ is nonzero is therefore open. For the smooth
    penalty take $S=[p]$ and $O_S$ to be the whole design space.
    On $O_S$, the implicit function theorem applied to the restricted
    KKT equations, whose coefficient matrix is
    $B_S+X_S^TDX_S\succ0$, makes $b^S$ a $C^1$ function of $X$.

    Write $I=S^c$, $b=b^S$, $h=b-\beta^*$, $u=Xh$,
    $v=\rho'(\eps-u)$, $r=\nabla R_S(b_S)$, and $q_I=X_I^Tv$.
    Use the map
    \[
        T_S:X\longmapsto(b_S,u,q_I).
    \]
    This contains the full KKT score, because $X_S^Tv=r$.
    On $O_S$, the event that $b^S$ is the full estimator is exactly
    the inactive KKT condition: $|q_j|\le\lambda$ for inactive
    elastic-net coordinates, or $\|q_{G_k}\|_2\le\lambda_k$ for
    inactive groups. Thus this event is determined by $T_S$.
    In particular, imposing it only selects output values of $T_S$;
    it introduces no additional boundary or truncation within a fiber.
    This restricted-fit construction also covers possible equality
    in an inactive KKT condition and requires no density assumption
    on the noise. When strict KKT holds, it agrees with working on
    the open, locally constant-support strata.

    We compute the normal Jacobian on $\{X\in O_S:h\ne0,v\ne0\}$.
    Differentiating the restricted KKT equation and the definitions
    of $u,q_I$ gives
    \begin{gather*}
        M_X\begin{pmatrix}db_S\\du\\dq_I\end{pmatrix}
        =L_X(dX),\\
        M_X=\begin{pmatrix}
            B_S&X_S^TD&0\\
            -X_S&I_n&0\\
            0&X_I^TD&I_{|I|}
        \end{pmatrix},\qquad
        L_X(dX)=\begin{pmatrix}dX_S^Tv\\dXh\\dX_I^Tv\end{pmatrix}.
    \end{gather*}
    Here $\det M_X=\det\mathcal S_S$.
    The map $L_X$ has rank $r_0=n+p-1$ and its adjoint has kernel
    spanned by $\ell=(h_S,-v,h_I)$.
    Its nonzero singular values give
    \[
        J_{r_0}L_X
        =\|h\|_2^{n-1}\|v\|_2^{p-1}\|\ell\|_2.
    \]
    For completeness, before permuting output coordinates,
    $L_XL_X^T$ has diagonal blocks $\|h\|_2^2I_n$ and
    $\|v\|_2^2I_p$ and off-diagonal block $vh^T$.
    Its nonzero eigenvalues are $\|h\|_2^2$ with multiplicity $n-1$,
    $\|v\|_2^2$ with multiplicity $p-1$, and
    $\|h\|_2^2+\|v\|_2^2$.

    The image of $T_S$ lies in the hypersurface
    \[
        \eta(b_S,u,q_I)
        =(b_S-\beta_S^*)^Tr- (\beta_I^*)^Tq_I-u^Tv=0,
        \qquad
        \nabla\eta=M_X^T\ell
        =\begin{pmatrix}B_Sh_S+r\\Du-v\\h_I\end{pmatrix}.
    \]
    The gradient is nonzero because $M_X$ is invertible.
    The volume distortion of $M_X^{-1}$ on $\ell^\perp$ is
    $|\det M_X|^{-1}\|M_X^T\ell\|_2/\|\ell\|_2$.
    Consequently
    \begin{equation}
        \label{regression-normal-jacobian}
        J_{r_0}T_S(X)
        =\frac{\|h\|_2^{n-1}\|v\|_2^{p-1}
                     \|\nabla\eta\|_2}
                    {\det\mathcal S_S(X)}.
    \end{equation}
    Every factor in the numerator depends only on $T_S(X)$.
    Coarea in local charts of this hypersurface therefore gives
    conditional density proportional to the original Gaussian
    density divided by \eqref{regression-normal-jacobian}, which
    proves \eqref{regression-density-tilt}.
    Every matrix satisfying $Xh=u$, $X^Tv=q$ satisfies the same
    restricted KKT equations, and, for a realized full-estimator
    output, also the same inactive KKT inequalities. By uniqueness
    it has the same estimator and belongs to the same $O_S$.
    Thus the fiber is the entire affine space stated in the lemma.
    The normalizing integral is positive and finite, since
    $\det\mathcal S_S$ is a strictly positive polynomial in the
    free design coordinates.
    There are finitely many candidate sets $S$, and conditioning
    on $h$ determines which one is used; no union-bound cost occurs.

    In the smooth case $B_S=\nabla^2R(\hat b)$.
    For elastic-net $B_S=n\nu I_{|S|}$, whereas for group-lasso plus
    ridge,
    \[
        B_S=n\nu I_{|S|}+\operatorname{blockdiag}_{k:\,G_k\subset S}
          \frac{\lambda_k}{\|\hat b_{G_k}\|_2}
          \left(I_{|G_k|}
          -\frac{\hat b_{G_k}\hat b_{G_k}^T}{\|\hat b_{G_k}\|_2^2}\right).
    \]
    Each is fixed on the fiber and is at least
    $n\nu I_{|S|}\succeq n\mu\Sigma_{S,S}$.
    \Cref{lemma:logdet-tilt} proves the asserted gradient bound.
    No upper bound on the group-penalty Hessian is needed.

    It remains to treat the exceptional ranks. If $h=0$, then
    $b=\beta^*$ and $v=\rho'(\eps)$ are deterministic. If $S$ is
    nonempty and $v\ne0$, even one active KKT equation
    $X_S^Tv=\nabla R_S(\beta_S^*)$ is a nontrivial Gaussian equality
    with a fixed right-hand side, an event of probability zero.
    If $v=0$, KKT implies $b=b_0$, the unique minimizer of $R$.
    Put $h_0=b_0-\beta^*$. The event $v=0$ is exactly
    $\rho'(\eps-Xh_0)=0$, which depends on $X$ only through $Xh_0$.
    Given $u=Xh_0$, conditioning further on this event changes
    nothing; the conditional law is the usual affine Gaussian law.
    If also $h_0=0$, it leaves the original Gaussian law unchanged.
    Finally, for an empty active set in the nonsmooth cases,
    $\hat A=0$ directly proves the stated zero-error assertion.
\end{proof}

\begin{remark}[Derivative formulas without a noise-density assumption]
    \label{remark:regression-active-derivatives}
    The restricted-fit argument also justifies the active matrices
    when applying \Cref{prop:mean_field} in regression.
    Fix the external noise and the signal projections $X\beta^*$
    (if $\beta^*=0$, no signal projection is needed).
    The responses are then fixed, and the remaining Gaussian
    coordinates have a density on their affine design slice.
    Strong convexity makes the full estimator locally Lipschitz on
    this slice. On $O_S$ the restricted estimator is $C^1$.
    The two estimators agree on the measurable event that the full
    active set is $S$, so their weak derivatives along the slice
    agree almost everywhere on this event. This is the locality
    property of weak derivatives, applied to their difference,
    which vanishes there.
    Differentiating the restricted KKT equations therefore gives
    exactly the projected derivative formula used in the proof of
    \Cref{prop:mean_field}, with the active matrix $\hat A$.
    This remains true if an inactive KKT equality persists on part
    of the slice. An empty active set gives a locally constant
    restricted estimator and zero derivative.
    This auxiliary use of signal projections is only for the
    mean-field derivative argument; the regression coarea map in
    \Cref{lemma:regression-coarea} still uses only the right-hand
    direction $\hat b-\beta^*$.
\end{remark}

\section{Disintegration for single-index models}
\label{sec:single-index-conditioning}

\begin{lemma}[The single-index conditional law]
    \label{lemma:single-index-coarea}
    Fix the external randomness $\eps$ in \eqref{single_index}.
    Let $w\ne0$ be deterministic and let $X$ have iid $N(0,\Sigma)$
    rows with $\Sigma$ positive definite. Suppose each $L_y$ is
    convex and $C^2$, with $0\le L_y''\le1$.
    Let $R$ be one of the penalties in
    \Cref{example:twice,example:enet,example:group}, with $R\in C^2$
    in the first case and $\nu\ge\mu\|\Sigma\|_{op}>0$.
    Use the smooth case for elastic-net with $\lambda=0$.
    Define
    \[
        a=Xw,\qquad b=\hat b,\qquad u=Xb,\qquad
        v=(-L_{y_i}'(u_i))_{i\in[n]},\qquad q=X^Tv,
    \]
    where $y_i=F(a_i,\eps_i)$. Let $S$ and $B_S$ be the active
    coordinate set and penalty Hessian as in
    \Cref{lemma:regression-coarea}, and put
    \[
        D=\diag((L_{y_i}''(u_i))_{i\in[n]}),\qquad
        \mathcal S_S=B_S+X_S^TDX_S.
    \]
    For almost every realized $(a,b,u,q)$ with $S\ne\varnothing$,
    $v\ne0$ and $b,w$ linearly independent, the conditional law
    $\pi_{a,b,u,q}$ satisfies
    \begin{equation}
        \label{single-index-density-tilt}
        \frac{d\pi_{a,b,u,q}}{dg_{a,b,u,q}}(X)
        =\frac{\det\mathcal S_S(X)}
        {\E_{g_{a,b,u,q}}[\det\mathcal S_S(X)]},
    \end{equation}
    where $g_{a,b,u,q}$ is the affine Gaussian law on
    \[
        \mathcal F_{a,b,u,q}
        =\{X:Xw=a,\ Xb=u,\ X^Tv=q\},
    \]
    with density proportional to
    $\exp(-\|X\Sigma^{-1/2}\|_F^2/2)$ with respect to Hausdorff
    measure on this affine space. The entire fiber has the same
    estimator, active set, $B_S$ and $D$. The logarithm of the
    density ratio in \eqref{single-index-density-tilt} has gradient
    of norm at most $1/\sqrt\mu$ in whitened free Gaussian coordinates.

    If $v=0$, the conditional law is affine Gaussian without a tilt.
    The event
    \[
        \{\,|S|\ge2,\ v\ne0,\ b\in\operatorname{span}(w)\,\}
    \]
    has probability zero. In the remaining collinear case
    $|S|=1$, the conditional law is again affine Gaussian.
    An empty active set gives $\hat A=0$ and zero weight error.
    No continuity or differentiability of the response link $F$
    is required beyond measurability.
\end{lemma}
\begin{proof}
    First disintegrate with respect to the linear statistic $a=Xw$.
    Its conditional law is Gaussian on
    $\mathcal A_a=\{X:Xw=a\}$, and $y=F(a,\eps)$ is now fixed.
    We apply coarea on this affine slice. This implements
    conditioning by the augmented map $(Xw,\hat b,X\hat b,q)$
    without differentiating the possibly discontinuous link $F$.

    For each candidate active set $S$, reuse the restricted
    estimator $b^S$ and the open set $O_S$ from the proof of
    \Cref{lemma:regression-coarea}, now relative to $\mathcal A_a$.
    The restricted fit is $C^1$ on $O_S$ by the implicit function
    theorem. Write $I=S^c$, $b=(b_S^S,0_I)$, $u=Xb$,
    $v=(-L_{y_i}'(u_i))_i$ and $r=\nabla R_S(b_S)$.
    The map on the slice is
    \[
        T_S:X\longmapsto(b_S,u,q_I),\qquad q_I=X_I^Tv.
    \]
    The active score is $q_S=r$. The event that the restricted fit
    is the full fit is determined by the inactive KKT inequalities
    on $q_I$. As in regression, these inequalities select output
    values of $T_S$, not subsets within one of its fibers.

    Consider first $v\ne0$ and $b,w$ linearly independent.
    A perturbation in the tangent space of $\mathcal A_a$ satisfies
    $dXw=0$. Differentiating the KKT equations gives
    \begin{gather*}
        M_X\,dT_S=L_X(dX),\\
        M_X=\begin{pmatrix}
            B_S&X_S^TD&0\\
            -X_S&I_n&0\\
            0&X_I^TD&I_{|I|}
        \end{pmatrix},\qquad
        L_X(dX)=\begin{pmatrix}dX_S^Tv\\dXb\\dX_I^Tv\end{pmatrix},
        \qquad dXw=0.
    \end{gather*}
    Again $\det M_X=\det\mathcal S_S>0$.
    The adjoint of this restricted $L_X$ has kernel spanned by
    the two independent vectors
    \[
        \ell_1=(b_S,-v,0_I),\qquad
        \ell_2=(w_S,0_n,w_I).
    \]
    Indeed, with $\Pi_w$ the orthogonal projection onto $w^\perp$,
    the adjoint equation, in the output ordering $(c_S,t,c_I)$, is
    $(vc^T+tb^T)\Pi_w=0$.
    Since $\Pi_wb\ne0$ and $v\ne0$, its solutions have
    $t=-\alpha v$ and $c=\alpha b+\beta w$.
    Thus $L_X$ and $DT_S$ have rank $r_0=n+p-2$.

    The image is contained in the codimension-two manifold
    determined by
    \[
        \eta_1=b_S^Tr-u^Tv=0,\qquad
        \eta_2=w_S^Tr+w_I^Tq_I-a^Tv=0.
    \]
    Their gradients with respect to $(b_S,u,q_I)$ are
    \[
        M_X^T\ell_1=
        \begin{pmatrix}B_Sb_S+r\\Du-v\\0_I\end{pmatrix},
        \qquad
        M_X^T\ell_2=
        \begin{pmatrix}B_Sw_S\\Da\\w_I\end{pmatrix}.
    \]
    They are independent and constant along each fiber, since
    $a,b,u,q_I$, and hence $D,B_S,r,v$, are fixed there.
    With $N=(\ell_1,\ell_2)$, the normal Jacobian on the slice is
    \begin{equation}
        \label{single-index-normal-jacobian}
        J_{r_0}T_S(X)
        =\frac{J_{r_0}L_X}{\det\mathcal S_S(X)}
        \left\{
        \frac{\det(N^TM_XM_X^TN)}{\det(N^TN)}
        \right\}^{1/2}
        =\frac{c(a,b,u,q_I)}{\det\mathcal S_S(X)}.
    \end{equation}
    The volume identity in the first equality follows by choosing
    an orthonormal basis of $\operatorname{span}(N)$ and completing
    it to an orthonormal basis of the output space: the volume
    distortion of $M_X^{-1}$ on its orthogonal complement is
    $|\det M_X|^{-1}$ times the volume distortion of $M_X^T$ on
    $\operatorname{span}(N)$.
    The factor $c$ is positive and constant on the fiber:
    $L_X$ on $\{dX:dXw=0\}$ depends only on $b,v,w$, and the
    columns of $M_X^TN$ were just shown to be fixed.

    Coarea on $\mathcal A_a$, in local charts of the above
    codimension-two manifold, divides the conditional Gaussian
    density by \eqref{single-index-normal-jacobian}, giving
    \eqref{single-index-density-tilt}.
    Conversely, every $X$ in $\mathcal F_{a,b,u,q}$ has the same
    response, predictions, active KKT equations and inactive KKT
    inequalities. Uniqueness of the minimizer gives the same
    estimator and hence the entire affine fiber, with no further
    selection boundary. The determinant is a positive polynomial
    on this fiber, so its Gaussian normalizing integral is finite.
    Disintegrating over $a$ and selecting the realized support
    gives the asserted joint conditional law.
    Since $B_S\succeq n\mu\Sigma_{S,S}$, the gradient bound follows
    directly from \Cref{lemma:logdet-tilt}; the extra right-hand
    projection cannot increase it.

    We now treat the exceptional ranks, still conditionally on $a$
    and $\eps$. If $v=0$, the KKT conditions force $b=b_0$, the
    deterministic minimizer of $R$. This event is precisely
    $L_{y_i}'((Xb_0)_i)=0$ for every $i$.
    Given $u=Xb_0$, the event is already determined. Consequently
    the conditional law given $(a,b_0,u,0)$ is the Gaussian law
    under the linear constraints $Xw=a$, $Xb_0=u$, omitting
    dependent or zero right-hand directions.

    Next suppose $v\ne0$, $S\ne\varnothing$ and $b=tw$.
    Necessarily $t$ is the unique minimizer
    \[
        t_0=\argmin_{t\in\R}
        \left\{\sum_iL_{y_i}(ta_i)+R(tw)\right\}.
    \]
    This is deterministic given $(a,\eps)$. Thus
    $b=t_0w$, $u=t_0a$, $v$, and the active score
    $r=\nabla R_S((t_0w)_S)$ are fixed.
    In a nonsmooth example a nonempty $S$ forces $t_0\ne0$, so
    $w_I=0$; in the smooth case $S=[p]$ and this is automatic.
    If $|S|\ge2$, choose a nonzero vector $c$ supported on $S$
    that is not proportional to $w$.
    Given $Xw=a$, the Gaussian scalar $v^TXc$ has variance
    \[
        \|v\|_2^2\left(
        c^T\Sigma c-\frac{(c^T\Sigma w)^2}{w^T\Sigma w}
        \right)>0.
    \]
    The active KKT equations would require $v^TXc=c_S^Tr$,
    an equality to a fixed number, which has probability zero.
    There are finitely many candidate supports, so this proves
    the claimed null event.

    If $|S|=1$, write $S=\{j\}$. Since $w_I=0$ and $w_j\ne0$,
    the entire active design column $X_j=a/w_j$ is already fixed.
    The restricted scalar fit $b_j=t_0w_j$ and its score $v$ are
    therefore fixed on the slice. It is the full fit exactly when
    the remaining linear scores $q_I=X_I^Tv$ satisfy the inactive
    KKT inequalities. Conditioning on these scores makes those
    inequalities redundant and gives an affine Gaussian law.
    The active determinant is constant on this fiber as well.
    This also covers the smooth case $p=1$, when $Xw$ determines
    the entire design. Finally, if $S=\varnothing$ in a nonsmooth
    example, $\hat A=0$ gives zero weight error without any coarea
    calculation.
\end{proof}

\bibliography{../bibliography/db}

\begin{thebibliography}{35}
\providecommand{\natexlab}[1]{#1}
\providecommand{\url}[1]{\texttt{#1}}
\expandafter\ifx\csname urlstyle\endcsname\relax
  \providecommand{\doi}[1]{doi: #1}\else
  \providecommand{\doi}{doi: \begingroup \urlstyle{rm}\Url}\fi

\bibitem[Auddy et~al.(2024)Auddy, Zou, Rahnamarad, and Maleki]{auddy2024approximate}
Arnab Auddy, Haolin Zou, Kamiar Rahnamarad, and Arian Maleki.
\newblock Approximate leave-one-out cross validation for regression with $\ell_1$ regularizers.
\newblock In \emph{International Conference on Artificial Intelligence and Statistics}, pages 2377--2385. PMLR, 2024.

\bibitem[Bayati and Montanari(2012)]{bayati2012lasso}
Mohsen Bayati and Andrea Montanari.
\newblock The lasso risk for gaussian matrices.
\newblock \emph{IEEE Trans. Inf. Theory}, 58\penalty0 (4):\penalty0 1997--2017, 2012.

\bibitem[Bellec(2023)]{bellec2020out_of_sample}
Pierre~C Bellec.
\newblock Out-of-sample error estimation for m-estimators with convex penalty.
\newblock \emph{Information and Inference: A Journal of the IMA}, 12\penalty0 (4):\penalty0 2782--2817, 2023.
\newblock URL \url{https://arxiv.org/pdf/2008.11840.pdf}.

\bibitem[Bellec(2025)]{bellec2022observable}
Pierre~C. Bellec.
\newblock {Observable adjustments in single-index models for regularized M-estimators with bounded p/n}.
\newblock \emph{Ann. Statist.}, 53\penalty0 (2):\penalty0 531 -- 560, 2025.
\newblock \doi{10.1214/24-AOS2464}.
\newblock URL \url{https://doi.org/10.1214/24-AOS2464}.

\bibitem[Bellec and Shen(2022)]{bellec2021derivatives}
Pierre~C Bellec and Yiwei Shen.
\newblock Derivatives and residual distribution of regularized m-estimators with application to adaptive tuning.
\newblock In \emph{Conference on Learning Theory}, pages 1912--1947. PMLR, 2022.
\newblock URL \url{https://proceedings.mlr.press/v178/bellec22a/bellec22a.pdf}.

\bibitem[Bellec and Tsybakov(2017)]{bellec2016bounds}
Pierre~C Bellec and Alexandre~B Tsybakov.
\newblock Bounds on the prediction error of penalized least squares estimators with convex penalty.
\newblock In \emph{Modern Problems of Stochastic Analysis and Statistics, Selected Contributions In Honor of Valentin Konakov}. Springer, 2017.
\newblock URL \url{https://arxiv.org/pdf/1609.06675.pdf}.

\bibitem[Bellec and Zhang(2021)]{bellec_zhang2018second_stein}
Pierre~C. Bellec and Cun-Hui Zhang.
\newblock Second-order stein: Sure for sure and other applications in high-dimensional inference.
\newblock \emph{Ann. Statist.}, 49\penalty0 (4):\penalty0 1864--1903, 2021.
\newblock ISSN 0090-5364.
\newblock URL \url{https://arxiv.org/pdf/1811.04121.pdf}.

\bibitem[Bellec and Zhang(2022)]{bellec_zhang2019debiasing_adjust}
Pierre~C. Bellec and Cun-Hui Zhang.
\newblock De-biasing the lasso with degrees-of-freedom adjustment.
\newblock \emph{Bernoulli}, 28\penalty0 (2):\penalty0 713--743, 2022.
\newblock ISSN 1350-7265.
\newblock URL \url{https://arxiv.org/pdf/1902.08885.pdf}.

\bibitem[Bellec and Zhang(2023)]{bellec_zhang2019second_poincare}
Pierre~C. Bellec and Cun-Hui Zhang.
\newblock Debiasing convex regularized estimators and interval estimation in linear models.
\newblock \emph{Ann. Statist.}, 51\penalty0 (2):\penalty0 391--436, 2023.
\newblock ISSN 0090-5364.
\newblock URL \url{https://arxiv.org/pdf/1912.11943.pdf}.

\bibitem[Bogachev(1998)]{bogachev1998gaussian}
Vladimir~Igorevich Bogachev.
\newblock \emph{Gaussian measures}.
\newblock Number~62. American Mathematical Soc., 1998.

\bibitem[Boucheron et~al.(2013)Boucheron, Lugosi, and Massart]{boucheron2013concentration}
St{\'e}phane Boucheron, G{\'a}bor Lugosi, and Pascal Massart.
\newblock \emph{Concentration Inequalities: A Nonasymptotic Theory of Independence}.
\newblock Oxford University Press, 2013.

\bibitem[Celentano and Montanari(2021)]{celentano2021cad}
Michael Celentano and Andrea Montanari.
\newblock Cad: Debiasing the lasso with inaccurate covariate model.
\newblock \emph{arXiv preprint arXiv:2107.14172}, 2021.

\bibitem[Celentano and Wainwright(2023)]{celentano2023challenges}
Michael Celentano and Martin~J Wainwright.
\newblock Challenges of the inconsistency regime: Novel debiasing methods for missing data models.
\newblock \emph{arXiv preprint arXiv:2309.01362}, 2023.

\bibitem[Celentano et~al.(2023)Celentano, Montanari, and Wei]{celentano2020lasso}
Michael Celentano, Andrea Montanari, and Yuting Wei.
\newblock The lasso with general gaussian designs with applications to hypothesis testing.
\newblock \emph{Ann. Statist.}, 51\penalty0 (5):\penalty0 2194--2220, 2023.

\bibitem[Davidson and Szarek(2001)]{DavidsonS01}
Kenneth~R Davidson and Stanislaw~J Szarek.
\newblock Local operator theory, random matrices and banach spaces.
\newblock \emph{Handbook of the geometry of Banach spaces}, 1\penalty0 (317-366):\penalty0 131, 2001.

\bibitem[Donoho and Montanari(2016)]{donoho2016high}
David Donoho and Andrea Montanari.
\newblock High dimensional robust m-estimation: Asymptotic variance via approximate message passing.
\newblock \emph{Probability Theory and Related Fields}, 166\penalty0 (3-4):\penalty0 935--969, 2016.

\bibitem[Donoho et~al.(2009)Donoho, Maleki, and Montanari]{donoho2009message}
David~L Donoho, Arian Maleki, and Andrea Montanari.
\newblock Message-passing algorithms for compressed sensing.
\newblock \emph{Proceedings of the National Academy of Sciences}, 106\penalty0 (45):\penalty0 18914--18919, 2009.

\bibitem[El~Karoui(2018)]{el_karoui2018impact}
Noureddine El~Karoui.
\newblock On the impact of predictor geometry on the performance on high-dimensional ridge-regularized generalized robust regression estimators.
\newblock \emph{Probability Theory and Related Fields}, 170\penalty0 (1-2):\penalty0 95--175, 2018.

\bibitem[El~Karoui et~al.(2013)El~Karoui, Bean, Bickel, Lim, and Yu]{el_karoui2013robust}
Noureddine El~Karoui, Derek Bean, Peter~J Bickel, Chinghway Lim, and Bin Yu.
\newblock On robust regression with high-dimensional predictors.
\newblock \emph{Proceedings of the National Academy of Sciences}, 110\penalty0 (36):\penalty0 14557--14562, 2013.

\bibitem[Koriyama et~al.(2025)Koriyama, Patil, Du, Tan, and Bellec]{koriyama2024precise}
Takuya Koriyama, Pratik Patil, Jin-Hong Du, Kai Tan, and Pierre~C Bellec.
\newblock Precise asymptotics of bagging regularized m-estimators.
\newblock \emph{Ann. Statist., in press.}, 2025.

\bibitem[Laurent and Massart(2000)]{laurent2000adaptive}
B.~Laurent and P.~Massart.
\newblock Adaptive estimation of a quadratic functional by model selection.
\newblock \emph{Ann. Statist.}, 28\penalty0 (5):\penalty0 1302--1338, 10 2000.
\newblock URL \url{http://dx.doi.org/10.1214/aos/1015957395}.

\bibitem[Loureiro et~al.(2021)Loureiro, Gerbelot, Cui, Goldt, Krzakala, Mezard, and Zdeborov{\'a}]{loureiro2021capturing}
Bruno Loureiro, Cedric Gerbelot, Hugo Cui, Sebastian Goldt, Florent Krzakala, Marc Mezard, and Lenka Zdeborov{\'a}.
\newblock Learning curves of generic features maps for realistic datasets with a teacher-student model.
\newblock \emph{Advances in Neural Information Processing Systems}, 34:\penalty0 18137--18151, 2021.

\bibitem[Miolane and Montanari(2021)]{miolane2018distribution}
L\'{e}o Miolane and Andrea Montanari.
\newblock The distribution of the {L}asso: uniform control over sparse balls and adaptive parameter tuning.
\newblock \emph{Ann. Statist.}, 49\penalty0 (4):\penalty0 2313--2335, 2021.
\newblock ISSN 0090-5364.
\newblock URL \url{https://doi.org/10.1214/20-aos2038}.

\bibitem[Montanari(2018)]{montanari2018mean}
Andrea Montanari.
\newblock Mean field asymptotics in high-dimensional statistics: From exact results to efficient algorithms.
\newblock In \emph{Proceedings of the International Congress of Mathematicians: Rio de Janeiro 2018}, pages 2973--2994. World Scientific, 2018.

\bibitem[Rad and Maleki(2020)]{rad2018scalable}
Kamiar~Rahnama Rad and Arian Maleki.
\newblock A scalable estimate of the out-of-sample prediction error via approximate leave-one-out cross-validation.
\newblock \emph{Journal of the Royal Statistical Society: Series B (Statistical Methodology)}, 82\penalty0 (4):\penalty0 965--996, 2020.

\bibitem[Rad et~al.(2020)Rad, Zhou, and Maleki]{rad2020error}
Kamiar~Rahnama Rad, Wenda Zhou, and Arian Maleki.
\newblock Error bounds in estimating the out-of-sample prediction error using leave-one-out cross validation in high-dimensions.
\newblock In \emph{International Conference on Artificial Intelligence and Statistics}, pages 4067--4077. PMLR, 2020.

\bibitem[Simchowitz et~al.(2018)Simchowitz, El~Alaoui, and Recht]{simchowitz2018tight}
Max Simchowitz, Ahmed El~Alaoui, and Benjamin Recht.
\newblock Tight query complexity lower bounds for pca via finite sample deformed wigner law.
\newblock In \emph{Proceedings of the 50th Annual ACM SIGACT Symposium on Theory of Computing}, pages 1249--1259, 2018.

\bibitem[Stein(1981)]{stein1981estimation}
Charles~M Stein.
\newblock Estimation of the mean of a multivariate normal distribution.
\newblock \emph{The annals of Statistics}, pages 1135--1151, 1981.

\bibitem[Stojnic(2013)]{stojnic2013framework}
Mihailo Stojnic.
\newblock A framework to characterize performance of lasso algorithms.
\newblock \emph{arXiv preprint arXiv:1303.7291}, 2013.

\bibitem[Thrampoulidis et~al.(2018)Thrampoulidis, Abbasi, and Hassibi]{thrampoulidis2018precise}
Christos Thrampoulidis, Ehsan Abbasi, and Babak Hassibi.
\newblock Precise error analysis of regularized $ m $-estimators in high dimensions.
\newblock \emph{IEEE Trans. Inf. Theory}, 64\penalty0 (8):\penalty0 5592--5628, 2018.

\bibitem[Vershynin(2018)]{vershynin2018high}
Roman Vershynin.
\newblock \emph{High-dimensional probability: An introduction with applications in data science}, volume~47.
\newblock Cambridge university press, 2018.

\bibitem[Wang et~al.(2018{\natexlab{a}})Wang, Zhou, Lu, Maleki, and Mirrokni]{wang2018approximate_fast_tuning}
Shuaiwen Wang, Wenda Zhou, Haihao Lu, Arian Maleki, and Vahab Mirrokni.
\newblock Approximate leave-one-out for fast parameter tuning in high dimensions.
\newblock In \emph{International Conference on Machine Learning}, pages 5228--5237. PMLR, 2018{\natexlab{a}}.

\bibitem[Wang et~al.(2018{\natexlab{b}})Wang, Zhou, Maleki, Lu, and Mirrokni]{wang2018approximate}
Shuaiwen Wang, Wenda Zhou, Arian Maleki, Haihao Lu, and Vahab Mirrokni.
\newblock Approximate leave-one-out for high-dimensional non-differentiable learning problems.
\newblock \emph{arXiv preprint arXiv:1810.02716}, 2018{\natexlab{b}}.

\bibitem[Xu et~al.(2021)Xu, Maleki, Rad, and Hsu]{xu2019consistent}
Ji~Xu, Arian Maleki, Kamiar~Rahnama Rad, and Daniel Hsu.
\newblock Consistent risk estimation in moderately high-dimensional linear regression.
\newblock \emph{IEEE Trans. Inf. Theory}, 67\penalty0 (9):\penalty0 5997--6030, 2021.

\bibitem[Zou et~al.(2025)Zou, Auddy, Rad, and Maleki]{zou2024theoretical}
Haolin Zou, Arnab Auddy, Kamiar~Rahnama Rad, and Arian Maleki.
\newblock Theoretical analysis of leave-one-out cross validation for non-differentiable penalties under high-dimensional settings.
\newblock In Yingzhen Li, Stephan Mandt, Shipra Agrawal, and Emtiyaz Khan, editors, \emph{Proceedings of The 28th International Conference on Artificial Intelligence and Statistics}, volume 258 of \emph{Proceedings of Machine Learning Research}, pages 4033--4041. PMLR, 03--05 May 2025.
\newblock URL \url{https://proceedings.mlr.press/v258/zou25b.html}.

\end{thebibliography}
\bibliographystyle{plainnat}

\end{document}